\documentclass[12pt]{article}
\usepackage[utf8]{inputenc}
\usepackage[T1]{fontenc}
\usepackage{epsfig}
\usepackage{amsfonts}
\usepackage{amssymb}
\usepackage{amsmath}
\usepackage{amsthm}
\usepackage{mathtools}
\usepackage{latexsym}
\usepackage{graphicx}
\usepackage{bm}
\usepackage{tabularx}
\usepackage{booktabs} 
\usepackage{enumerate}
\usepackage{caption}
\usepackage[titletoc]{appendix}
\usepackage[numbers]{natbib}
\usepackage{bbm}
\usepackage{url}
\usepackage{subfig}
\usepackage{hyperref}
\usepackage{color}
\usepackage[all]{xy}
\usepackage{float}
\usepackage{mwe}
\usepackage{mathtools}
\usepackage{media9}

\DeclarePairedDelimiter{\floor}{\lfloor}{\rfloor}

\catcode`\@=11

\newcommand{\shortdot}[1]{\raisebox{-0.4pt}{$\stackrel{\bullet}{#1}$}}
\newcommand{\updot}[1]{\raisebox{0.9pt}{$\stackrel{\bullet}{#1}$}} 

\DeclareBoldMathCommand\boldlangle{\left\langle} 
\DeclareBoldMathCommand\boldrangle{\right\rangle}

\theoremstyle{plain}
\newtheorem{theorem}{Theorem}[section]
\newtheorem{lemma}[theorem]{Lemma}

\newtheorem{proposition}[theorem]{Proposition}

\theoremstyle{definition}

\theoremstyle{remark}

\begin{document}

\title{The Impact of Smartphone Apps on Bike Sharing Systems}
\author{ 
  Shuang Tao \\ School of Operations Research and Information Engineering \\ Cornell University
\\ 293 Rhodes Hall, Ithaca, NY 14853 \\  st754@cornell.edu  \\ 
\and
  Jamol Pender \\ School of Operations Research and Information Engineering \\ Cornell University
\\ 228 Rhodes Hall, Ithaca, NY 14853 \\  jjp274@cornell.edu  \\ 
 }

\maketitle
\begin{abstract}
Bike-sharing systems are exploding in cities around the world as more people are adopting sustainable transportation solutions for their everyday commutes.  However, as more people use the system, riders often encounter that bikes or docks might not be available when they arrive to a station.  As a result, many systems like CitiBike and Divvy provide riders with information about the network via smartphone apps so that riders can find stations with available bikes.  However, not all customers have adopted the use of these smartphone apps for their station selection.  By combining customer choice modeling and finite capacity queueing models, we explore the impact of the smartphone app technology to increase throughput and reduce blocking in bike sharing systems.  To this end, we prove a mean-field limit and a central limit theorem for an empirical process of the number of stations with $k$ bikes.  We also prove limit theorems for a new process called the ratio process, which characterizes the proportion of stations whose bike availability ratio lies within a particular partition of the interval [0,1].  For the mean field limit, we prove that the equilibrium exists, is unique, and that the stationary distribution of the empirical measure converges to a Dirac mass at the same equilibrium, thus showing an interchange of limits result ($\lim_{t\rightarrow \infty}\lim_{N\rightarrow \infty}=\lim_{N\rightarrow \infty}\lim_{t\rightarrow \infty}$).  Our limit theorems provide insight on the mean, variance, and sample path dynamics of large scale bike-sharing systems.  Our results illustrate that if we increase the proportion of customers that use smartphone app information, the entropy of the bike sharing network is reduced, and riders experience less blocking in the network.

\end{abstract}

%

\section{Introduction}

Bike sharing is expanding rapidly around the world.  Around 1000 cities worldwide have implemented bike sharing programs and the majority of them are achieving great success.  Bike sharing programs aim to provide affordable short-term trips for people living in urban areas as an alternative to other public transportation or private vehicles. They have the benefits of reducing traffic congestion, noise, and mitigating air pollution. Studies also report observing an increase in cycling and health benefits where bike sharing systems are run.

Bike sharing has developed significantly since its initial launch in Amsterdam in 1965, and it is generally considered that bike sharing has gone through three generations of major changes in the implementation and design of the system. The contemporary 3rd generation schemes, which are also called information technology (IT) based systems, exploited information and communication technology to effectively automate the bike sharing system. This includes automated smartcards (or magnetic stripe card), electronic bike locking, electronic payment, GPS (Global Positioning System) tracking of bikes, and networked self-service stations managed by central computer systems and technologies like Radio Frequency Identification (RFID) to monitor the location of bikes in the system.  Many systems also introduced the use of websites or smartphone apps to provide users with real-time information about available bikes and open docks at nearby stations.

The biggest issue for all bike sharing systems is the scarcity of resources to move all riders around each city at all times of the day. That is, customers can encounter empty stations when picking up bikes, and they can also encounter full stations when dropping off bikes. In both situations, customers get blocked and therefore system throughput is reduced and customers are frustrated and might leave the system altogether. Some empirical and analytical studies of bike availability at CitiBike stations show that during peak hours, over 30\% of the stations are empty and 8\% of the stations are full \citet{faghih2015empirical, freund2017minimizing, ghosh2017dynamic, hampshire2012analysis, kabra2016bike, jian2016simulation, tao2017stochastic}. The pressing need to come up with economic yet effective solutions to this issue is there. 

The most direct way is real-time rebalancing of bikes across stations. Examples of papers that analyze real-time rebalancing strategies are \citet{erdougan2015exact, contardo2012balancing, freund2017minimizing, ghosh2017dynamic, nair2014equilibrium, o2015smarter, o2015data, schuijbroek2017inventory, raviv2013static, faghih2017empirical, singla2015incentivizing, freund2020data}.  However, in practice, real-time rebalancing is both expensive and time-consuming, and it has not been adapted by most large scale bike sharing systems. The main reason is that trucks need to pick up bikes from nearly full stations and deliver them to empty stations during high demand times. Since peak hours in a workday only last about 1-2 hours, real-time rebalancing may not be able to solve unbalanced inventory in bike sharing systems.  Another solution is to implement real-time dynamic pricing, see for example \citet{fricker2016incentives, kabra2016bike, pfrommer2014dynamic, waserhole2013pricing, waserhole2016pricing, chemla2013self}.  However, given the majority of the users of CitiBike are subscribers who pay an annual fee to get unlimited rides, real-time pricing will not make a big difference since it only affects non-subscribers.

With the advances in technology, many large scale bike sharing systems such as CitiBike now give their users the ability to see real-time bike/dock availability at nearby stations through a smartphone app and web API.  A picture of the CitiBike app is given in Figure \ref{Fig_screenshot}.   This information helps users make better decisions such as where to pick up bikes from nearby locations. This means that when stations are empty or have very few bikes, users who have this information through smartphone app will be less likely to pick up bikes at these stations. Similarly, when stations are full or nearly full, users who have the information are more likely to pick up bikes at these stations. Intuitively, providing users real-time information about bike/dock availability should have the power of guiding users to move bikes from stations with more bikes to stations with less bikes and therefore making the system more balanced.  However, it is not known how the app itself affects the underlying network dynamics since this additional information can change the behavior of the users.  

\begin{figure}[H]
\centering
\includegraphics[scale=0.12]{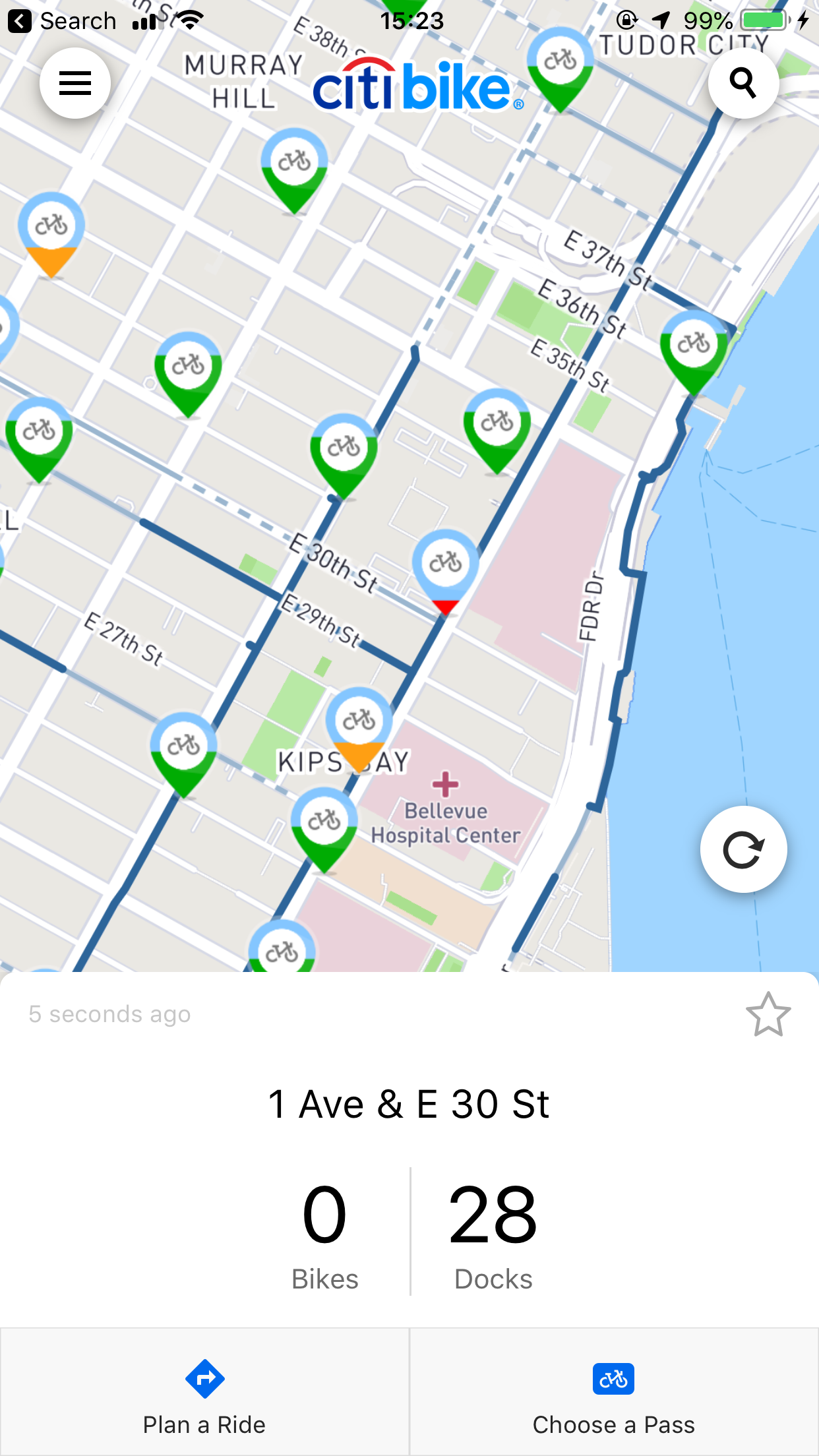}~
\includegraphics[scale=0.12]{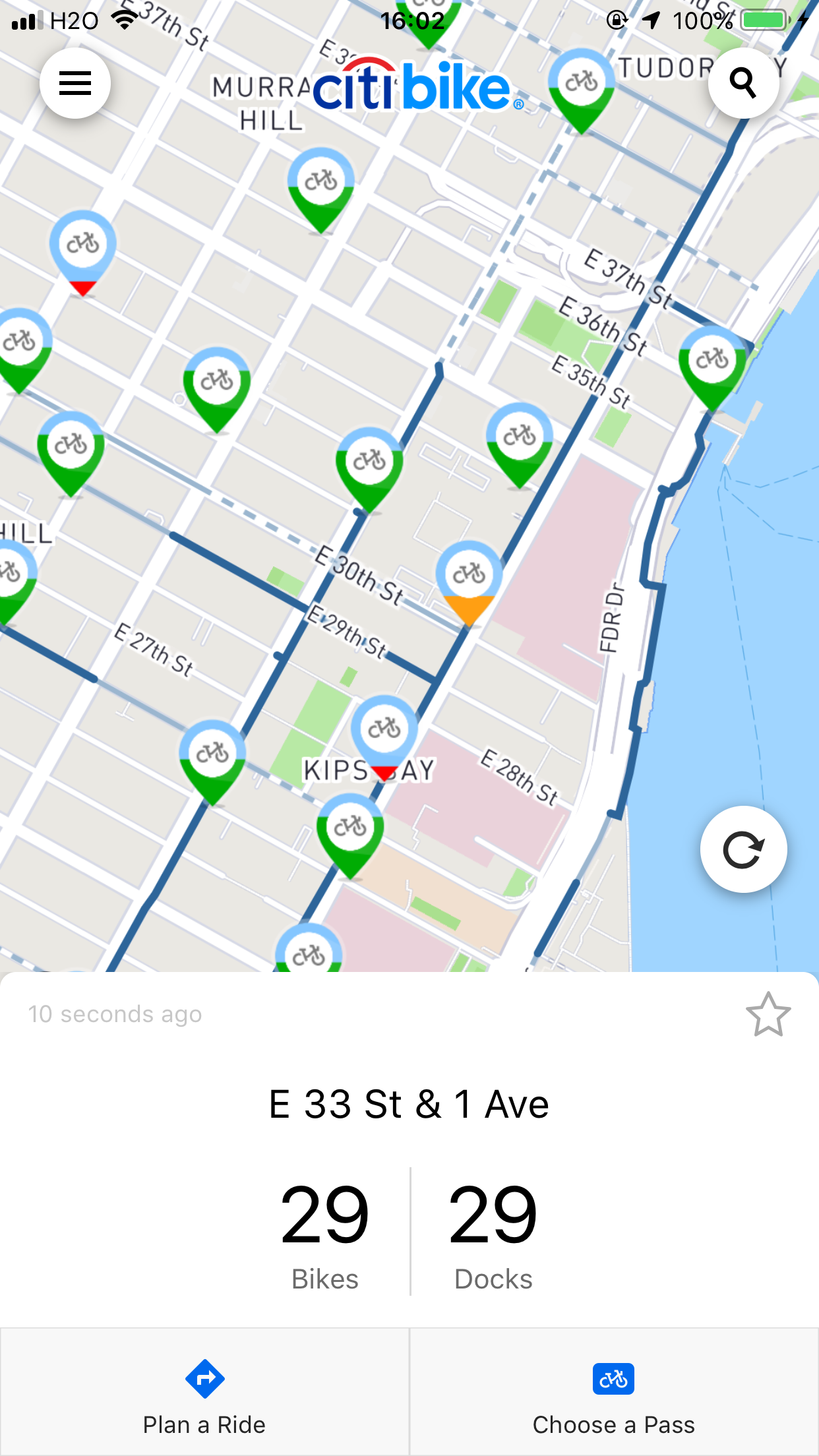}
\caption{CitiBike smartphone app screenshots}\label{Fig_screenshot}
\end{figure}

In this paper, we study the influence of smartphone app and its power to guide users to non-empty stations. Inspired by the empirical work of \citet{kabra2016bike, nirenberg2018impact} and the theoretical work of \citet{pender2017queues,
pender2018analysis, novitzky2019nonlinear, novitzky2019limiting}, our approach combines customer choice models and finite capacity queueing models.  However, our work is more analytical in our approach since we prove a mean-field and central limit theorems for the empirical process of the proportion of stations with $k$ bikes. These limit theorems allow us to study the impact of information by analyzing the asymptotic behavior of the mean, variance and sample path dynamics of large scale bike sharing systems under different choice models and different parameters, such as the proportion of customers that use the smartphone app information, customers' sensitivity to information, etc.   Moreover, we also develop a mean-field and central limit theorems for the bike availability ratio process, which counts the proportion of stations whose bike availability falls into an interval $[i/n,(i+1)/n)$. This ratio process is especially useful when station capacities are not uniform and identical. For example, in the case when we have two stations with capacity equal to 5 and 70 respectively. Having 3 bikes at station 1 means it already achieves over half of the capacity, while at station 2 it clearly means too few bikes. In such case the ratio process will provide more accurate information about the condition of a station than the empirical process.  Finally, for both the empirical process and the ratio process, we study the system in steady state and show that an equilibrium exists and is unique.  We also show that the steady state empirical process and ratio process converge to a Dirac mass at the equilibrium point, thus proving an interchanging of limits result.


 \subsection{Main Contributions of Paper}

The contributions of this work can be summarized as follows:    
\begin{itemize}
\item We prove a transient mean field and central limit theorem for the empirical process.   
\item We construct a new process called the ratio process and derive a mean field and central limit theorem for the ratio process and compute its steady state equilibrium.    
\item We prove an interchange of limits for the mean field limit empirical process and the ratio process.  We show that the limit is a Dirac mass.  
\item We show that our stochastic choice queueing model can replicate empirical process and ratio process given by the data of the CitiBike network.
\end{itemize} 


\subsection{Organization of Paper}

The remainder of this paper is organized as follows. Section \ref{Sec_Bike_Model} introduces our bike sharing model with customer choice.  We also prove a mean field and central limit theorem for the empirical process for the proportion of stations with $k$ bikes.  In Section \ref{Sec_Ratio},  we also prove a mean field and central limit theorem for the empirical process for the proportion of stations  whose bike availability falls into a interval $[i/n,(i+1)/n)$.  In Section \ref{Sec_SteadyState}, we compute explicitly the equilibrium and Lyapunov functions of the mean field limit of the empirical process and the ratio process.  Finally, we show that both the empirical process and the ratio process satisfy an interchange of limits.   In Section \ref{Sec_simulation}, we discuss the simulation results of our model, in both stationary and non-stationary arrival processes settings.  These results incorporates  real-life data from CitiBike, and provide insights into how different aspects of the impact of information affect the behavior of empirical measure over time.  In Section \ref{conc}, we give concluding remarks and provide some future directions of research. Finally, in Section \ref{App} the Appendix, additional proofs and theorems are given.


\section{Stochastic Model and Limit Theorems} \label{Sec_Bike_Model}

In this section, we construct a Markovian bike-sharing queueing model where customers can pick-up and drop-off bikes at each station if there is available capacity.

\subsection{Bike Sharing Model with Customer Choice Dynamics}
Motivated by the CitiBike system, we consider a bike sharing system with $N$ stations and a fleet of $M$ bikes in total. We assume that the arrival of customers to the stations are independent Poisson processes with rate $\lambda_i$ for the  $i^{th}$ station.  When a customer arrives at a station, if there is no bike available, he will then leave the system and is immediately blocked and lost. Otherwise, he takes a bike and rides to station $j$ with probability $P_j$.  We assume that the travel time of the rider is exponentially distributed with mean $1/\mu$, for every transition from one station to another.  Since we are concerned with finite capacity stations, we assume that the $i^{th}$ station has a bike capacity of $K_i$, which is assumed to be finite for all stations.   Thus, when a customer arrives at station $j$, if there are less than $K_j$ bikes in this station, he returns his bike and leaves the system. If there are exactly $K_j$ bikes (i.e. the station is full), the customer randomly chooses another station $k$ with probability $P_k$ and goes to that station to drop the bike off.  As before, it takes a time that is exponentially distributed with mean $1/\mu$. Finally, the customer rides like this again until he can return his bike to a station that is not full. 

Below we provide Table \ref{table:nonlin} for the reader's convenience to so that they understand the notation that we will use throughout the rest of the paper.  

\begin{table}[ht]
\caption{Notation} 
\centering 
\begin{tabular}{c c } 
\hline\hline 

\hline 
$N$ & Number of stations \\ 
$M$ & Number of bikes \\
$K_{i}$ & Capacity at station $i$\\
$\lambda_i$ & Arrival rate at station $i$\\
$1/\mu$ & Mean travel time\\
$P_i$ & Routing probability to station $i$\\
$X_i(t)$ & Number of bikes at station $i$ at time $t$\\
$\gamma$ & Average number of bikes at each station\\
$p$ & The fraction of users that have information about bike distribution \\
$g$ & The choice function\\
$\theta$ & Users' sensitivity to information (exponential choice function)\\
$c$ & Users' sensitivity threshold to information (minimum choice function)\\
$\alpha$ & Users' sensitivity to information (polynomial choice function)\\
\hline 
\end{tabular}
\label{table:nonlin} 
\end{table}

Throughout the rest of the paper, we assume without loss of generality that the arrival rate $\lambda_i=\lambda$ and service rate $\mu=1$, and that the routing probability from each station is uniform i.e. $P_i=1/N$. We also assume that capacities across all stations are equal, i.e. $K_i=K, i=1,...,N$. We refer to \citet{tao2017stochastic} on how our model can be extended seamlessly to settings with non-uniform arrival rate, routing probabilities and capacities. With our notation in hand, we are ready to develop our stochastic model for our bike-sharing network.

 Let $\mathbf{X}(t)=(X_{1}(t),\cdots, X_{N}(t))$, where $X_{i}(t)$ is the number of bikes at station $i$ at time $t$, then $X_{i}(t)$ is a continuous time Markov chain (CTMC), in particular a state dependent $M/M/1/K_{i}$ queue. In this model, the rate of dropping off bikes at station $i$ is equal to $\mu P_{i}\left(M-\sum_{k=1}^{N}X_{k}(t)\right)\mathbf{1}\{X_{i}(t)<K_{i}\}$, and the rate of retrieving bikes at station $i$ is equal to $\left((1-p)\lambda+p\lambda_{i}N\frac{g(X_i(t))}{\sum_{j=1}^{N}g(X_{j}(t))}\right)\mathbf{1}\{X_{i}(t)>0\}$, where $g$ is an arbitrary non-decreasing function that is also non-negative.
 
 \begin{proposition}\label{functional_forward_choice}
 For any continuous  function $f$, the CTMCs $\{X_i(t)\}_{i=1}^{N}$ satisfy the following functional forward equations,
 \begin{eqnarray}
& & \lefteqn{ \updot{\mathbb{E}}[f(X_i(t)) | X_i(0) = x_i]  \equiv  \frac{d}{dt} \mathbb{E}[f(X_i(t)) | X_i(0) = x_i]  } \nonumber \\ &=&  \mathbb{E}\left[ \left( f(X_i(t) +1) - f(X_i(t)) \right) \cdot \left(\mu P_{i}\left(M-\sum_{k=1}^{N}X_{k}(t)\right)\mathbf{1}\{X_{i}(t)<K_{i}\}\right) \right]  \nonumber\\& +&  \mathbb{E}\left[ \left( f(X_i(t) -1) - f(X_i(t)) \right) \cdot \left( \left((1-p)\lambda +p\lambda N\frac{g(X_i(t))}{\sum_{j=1}^{N}g(X_{j}(t))}\right)\mathbf{1}\{ X_{i}(t) > 0 \}  \right) \right]. \nonumber \\
 \end{eqnarray} 
 \end{proposition}

\subsection{Empirical Bike Process}
Although the functional forward equations given in Proposition \ref{functional_forward_choice} describe the exact dynamics of the mean of the bike sharing system, the system of differential equations are not closed.  This non-closure property of the functional forward equations in this model arises from the fact that the bike sharing system has finite capacity.  More importantly, it also implies that we need to know a priori the full distribution of the whole stochastic process $\mathbf{X}(t)$ in order to calculate the mean or variance or any moment for that matter. Finally if we even wanted to solve these equations and knew the entire distribution of $\mathbf{X}(t)$ a priori, there are still $O(N^2)$ differential equations that would need to be numerically integrated. For some major bike sharing networks such as CitiBike in NYC, $N$ is around 900; for Divvy in Chicago, $N$ is around 600. This can be very computationally expensive and thus, we must take a  different approach to analyze our bike sharing system. 
 
Moreover, if we want to analyze the limiting behavior of $\{X_{i}\}_{i=1}^{N}$ as  CTMCs, as we let $N$ go to $\infty$, the mean field limit would become infinite dimensional, which is quite complicated.  However, if we instead analyze an empirical measure process for $\mathbf{X}(t)$, we can use the finite capacity nature of the bike sharing system to our advantage and have a finite dimensional CTMC for the empirical measure process.  See \citet{tao2017stochastic} for more discussion on this topic.

Thus, we construct an empirical measure process that counts the proportions of stations with $n$ bikes.  This empirical measure process is given below by the following equation: 
\begin{equation}
Y_{t}^{N}(n)=\frac{1}{N}\sum_{i=1}^{N}\mathbf{1} \{X_i(t)=n\}.
\end{equation}

Conditioning on $Y_t^{N}(n)=y(n)$, the transition rates of $y$ are specified as follows:
\begin{itemize}
\item When there is an customer arrival to a station with $n$ bikes and relative utilization $r$ to retrieve a bike, the
proportion of stations having $n$ bikes goes down by $1/N$, the proportion of stations having $n-1$ bikes goes up by $1/N$, and the transition rate $Q^N$ is 
\begin{eqnarray}
Q^{N}\left(y,y+\frac{1}{N}(\mathbf{1}_{n-1}-\mathbf{1}_{n}) \right) &=&  \left((1-p)\lambda Ny(n)+p\lambda N\frac{y(n)g(n)}{\sum_{j=0}^{K}y(j)g(j)}\right) \mathbf{1}_{n>0}.\nonumber\\
\end{eqnarray}
\item When there is a customer returning a bike to a station with $n$ bikes and relative utilization $r$, the proportion of stations having $n$ bikes goes down by $1/N$, the proportion of stations having $n+1$ bikes goes up by $1/N$, and the transition rate $Q^N$ is 
\begin{eqnarray}
Q^{N} \left(y,y+\frac{1}{N}(\mathbf{1}_{n+1}-\mathbf{1}_{n}) \right) &=&
 y(n) \cdot \mu \cdot \left(M-\sum_{n'}n'y(n')N\right)\mathbf{1}_{n<K}\nonumber  \\  &=&
 y(n)N\left(\frac{M}{N}-\sum_{n'}n'y(n')\right)\mathbf{1}_{n<K}.
\end{eqnarray}
\end{itemize}

Similarly, we have the functional forward equations for $Y^{N}(t)$ as follows. 
\begin{proposition}\label{functional_forward_eqn}
 For any continuous function $f$, $Y^N(t)$ satisfies the following functional forward equation,
 \begin{eqnarray}
 & &\lefteqn{ \updot{\mathbb{E}}[f(Y^{N}_{n}(t)) | Y^{N}_{n}(0) = y_{n}(0)]  \equiv  \frac{d}{dt} \mathbb{E}[f(Y^{N}_{n}(t)) | Y^{N}_{n}(0) = y_{n}(0)]  } \nonumber \\
  &=&  \mathbb{E}\left[ \left( f\left(Y^{N}_{n}(t) +1/N\right) - f(Y^{N}_{n}(t)) \right) \cdot \left(Y_{n+1}^{N}(t)N\lambda\left((1-p) +p\frac{g(n+1)}{\sum_{j=0}^{K}Y_{j}^{N}(t)g(j)}\right) \mathbf{1}_{n<K}\right.\right.\nonumber\\
  & & \left.\left.+Y^{N}_{n-1}(t)N\left(\frac{M}{N}-\sum_{n'}n'Y_{n'}^{N}(t)\right)\mathbf{1}_{n>0}\right) \right]  \nonumber\\
  & & +\mathbb{E}\left[ \left( f(Y^{N}_{n}(t) -1/N) - f(Y^{N}_{n}(t)) \right) \cdot \left( Y_{n}^{N}(t)N\lambda\left((1-p) +p\frac{g(n)}{\sum_{j=0}^{K}Y_{j}^{N}(t)g(j)}\right) \mathbf{1}_{n>0}\right.\right.\nonumber \\
  & &\left.\left.+Y^{N}_{n}(t)N\left(\frac{M}{N}-\sum_{n'}n'Y_{n'}^{N}(t)\right)\mathbf{1}_{n<K} \right) \right],
 \end{eqnarray} 
 for $n=0,1,...,K$.
 \end{proposition}

However, we should mention that the functional forward equations in Proposition \ref{functional_forward_eqn} are also not closed. Therefore, we need to use another method for the analysis of this model to get around the non-closure properties of the equations, which leads us to the mean field limit and diffusion limit analysis in the following sections.

\subsection{Mean Field Limit of Empirical Process} \label{Fluid_Limit}

In this section, we prove the mean field limit for our bike sharing model.  A mean field limit describes the large station dynamics of the bike sharing network over time.  Deriving the mean field limit allows us to gain insights on average system behaviors, when the the demand for bikes and the number of stations are very large. Thus, we avoid the need to study an $N$-dimensional CTMC and compute its steady state distribution in this high dimensional setting.

\begin{theorem}[Functional Law of Large Numbers]\label{fluid_limit}
Let $|. |$ denote the Euclidean norm in $\mathbb{R}^{K+1}$. Suppose that $\lim_{N\rightarrow \infty}\frac{M}{N}=\gamma$ and $Y_{0}^{N}\xrightarrow{p} y_0$,  then we have for $\forall \epsilon>0$
$$\lim_{N\rightarrow \infty}P\left(\sup_{t\leq t_0}|Y_t^N-y_t|>\epsilon\right)=0,$$
where $y_t$ is the unique solution to the following differential equation starting at $y_0$
\begin{equation}\label{diff_eqn}
\shortdot{y}=b(y),
\end{equation}
where $b:[0,1]^{K+1}\rightarrow \mathbb{R}^{K+1}$ is a vector field satisfies
\begin{eqnarray}\label{eqn:b}
b(y)&=&\sum_{n=0}^{K}\left[ \left((1-p)+p\frac{g(n)}{\sum_{j=0}^{K}y_{j}g(j)}\right)(\mathbf{1}_{n-1}-\mathbf{1}_{n})\lambda\mathbf{1}_{n>0}\right.\nonumber\\
& & \left.+\left(\gamma-\sum_{j=0}^{K} jy_{j} \right)(\mathbf{1}_{n+1}-\mathbf{1}_{n})\mathbf{1}_{n<K}\right] y_n,
\end{eqnarray}
or 
\begin{equation}
b(y_0)=\underbrace{-\left(\gamma-\sum_{j=0}^{K} jy_{j} \right)y_0}_{\text{return a bike to a no-bike station}}+\underbrace{\left((1-p)+p\frac{g(1)}{\sum_{j=0}^{K}y_{j}g(j)}\right)\lambda y_{1}}_{\text{retrieve a bike from a 1-bike station}},
\end{equation}
\begin{equation}
\begin{split}
b(y_k)=& \underbrace{\left((1-p)+p\frac{g(k+1)}{\sum_{j=0}^{K}y_{j}g(j)}\right)\lambda y_{k+1}}_{\text{retrieve a bike from a $k+1$-bike station}}-\underbrace{\left(\left((1-p)+p\frac{g(k)}{\sum_{j=0}^{K}y_{j}g(j)}\right)\lambda+\gamma-\sum_{j=0}^{K} jy_{j} \right)y_{k}}_{\text{retrieve and return a bike to a $k$-bike station}}\\
&+\underbrace{\left(\gamma-\sum_{j=0}^{K} jy_{j} \right)y_{k-1}}_{\text{return a bike to a $k-1$-bike station}},
\end{split}
\end{equation}
for $ k=1,...,K-1$, and 
\begin{equation}
b(y_K)=\underbrace{-\left((1-p)+p\frac{g(K)}{\sum_{j=0}^{K}y_{j}g(j)}\right)\lambda y_{K}}_{\text{retrieve a bike from a $K$-bike station}}+\underbrace{\left(\gamma-\sum_{j=0}^{K} jy_{j} \right)y_{K-1}}_{\text{return a bike to a $K-1$-bike station}}.
\end{equation}
\end{theorem}

\begin{proof}
The complete proof is provided in the Appendix (Section \ref{App}).
\end{proof}

Our mean field limit shows that the analysis of our bike sharing system is essentially solving a system of ordinary differential equations when the number of stations $N$ goes to infinity. Using the fluid limit, we can estimate important performance measures such as proportion of empty or full stations, what factors (such as fleet size and capacity) are affecting these performance measures, and what we can do to improve the system performance. Deriving the mean field limit also allows us to compute the system's equilibrium point, which is addressed in Section \ref{Sec_SteadyState}. 

However, knowing the mean field limit is not enough, we would like to know more about the stochastic variability of the system, i.e. the fluctuations around the mean field limit, and what is stochastic sample path behavior of the system.  Thus, in the subsequent section we develop a functional central limit theorem for our bike sharing model, and explain why it is important for understanding stochastic fluctuations.


\subsection{Diffusion Limit of Empirical Process}\label{Diffusion_Limit}
In this section, we derive the diffusion limit of our stochastic empirical process bike sharing model.  Diffusion limits are critical for obtaining a deep understanding of the sample path behavior of stochastic processes.  One reason is that diffusion limits describe the fluctuations around the mean or mean field limit and can help understand the variance or the asymptotic distribution of the stochastic process being analyzed, and help provide confidence intervals to the mean field limit.  We define our diffusion scaled bike sharing model by subtracting the mean field limit from the scaled stochastic process and rescaling by $\sqrt{N}$.  Thus, we obtain the following expression for the diffusion scaled bike sharing empirical process
\begin{equation}
D_{t}^{N}=\sqrt{N}(Y_{t}^{N}-y_{t}) .
\end{equation}

Unlike many other ride-sharing systems such as Lyft or Uber, bike sharing does not have a pricing mechanism to redistribute bikes in real-time to satisfy demand.  For this reason, it is essential to understand the dynamics and behavior of $D_{t}^{N}$.  $D_{t}^{N}$ can be useful for describing the probability that the proportion of stations with $i$ bikes exceeds a threshold i.e. $\mathbb{P}(Y^N_t > x)$ for some $x \in [0,1]^{K+1}$.  It also describes this probability in a situation where there is no control or rebalancing of bikes in the system.  This knowledge of the uncontrolled system is especially important for newly-started bike sharing systems who are still in the process of gathering information about the system demand. Having the diffusion limit help managers of these bike sharing systems understand the system dynamics and  stability, which in terms help them make long term managerial decisions. It is also helpful in the case when the operators of the bike sharing system have no money for rebalancing the system to meet  real-time demand.

Using the semi-martingale decomposition of $Y^N_t$ given in Equation (\ref{semiY}), we can write a similar decomposition for $D_{t}^{N}$ as follows:
\begin{equation}
\begin{split}
D_{t}^{N}&=\sqrt{N}(Y_{0}^{N}-y_0)+\sqrt{N}M_{t}^{N}+\int_{0}^{t}\sqrt{N}[\beta(Y_{s}^{N})-b(y_s)]ds\\
&=D_{0}^{N}+\sqrt{N}M_{t}^{N}+\int_{0}^{t}\sqrt{N}[\beta(Y_{s}^{N})-b(Y_{s}^{N})]ds+\int_{0}^{t}\sqrt{N}[b(Y_{s}^{N})-b(y_s)]ds.
\end{split}
\end{equation}

Define 
\begin{equation}\label{sde}
D_{t}=D_{0}+\int_{0}^{t}b'(y_s)D_s ds+M_t,
\end{equation}
where $b'(y)=\left(\frac{\partial b(y(i))}{\partial y(j)}\right)_{ij}$ and $M_{t}=(M_{t}(0),\cdots,M_{t}(K))\in \mathbb{R}^{K+1}$ be a real continuous centered Gaussian martingale, with Doob-Meyer brackets given by
\begin{eqnarray}
\boldlangle M(k) \boldrangle_{t}&=&\int_{0}^{t}(b_{+}(y_{s})(k)+b_{-}(y_{s})(k))ds,\nonumber\\
\boldlangle M(k),M(k+1) \boldrangle_{t}&=&-\int_{0}^{t}\left[\left(\lambda(1- p)+p\frac{g(k+1)}{\sum_{j=0}^{K}y_{j}g(j)}\right) y_{k+1}\right.\nonumber\\
& & + \left.\left(\gamma-\sum_{j}jy_{j}\right) y_{k}\right]ds \quad \text{for } k<K,\nonumber \\
\boldlangle M(k),M(j) \boldrangle_{t}&=&0\quad \text{for } |k-j|>1.
\end{eqnarray}
Here $b_{+}(y), b_{-}(y)$ denote the positive and the negative parts of function $b(y)$ respectively.

Now we state the functional central limit theorem for the empirical measure process below.
\begin{theorem}[Functional Central Limit Theorem]\label{difftheorem}
Consider $D_{t}^{N}$ in $\mathbb{D}(\mathbb{R}_{+},\mathbb{R}^{K+1})$ with the Skorokhod $J_{1}$ topology, and suppose that 
$\limsup_{N\rightarrow \infty}\sqrt{N}(\frac{M}{N}-\gamma)< \infty$. 
Then if $D_{0}^{N}$ converges in distribution to $D_{0}$, then $D_{t}^{N}$ converges to the unique OU process solving 
$$D_{t}=D_{0}+\int_{0}^{t}b'(y_s)D_s ds+M_t$$
 in distribution, where $b'(y)$ is specified as follows,

\begin{equation}\label{b'1}
\frac{\partial}{\partial y_{i}}b(y_0)=iy_{0}-\left(\gamma-\sum_{j=0}^{K}jy_{j}\right)\mathbf{1}_{i=0}+\lambda\left[(1-p)+p\frac{g(1)}{\sum_{j=0}^{K}y_{j}g(j)}\right]\mathbf{1}_{i=1}-\frac{\lambda  p y_1 \cdot g(1)g(i)}{\left(\sum_{j=0}^{K}y_{j}g(j)\right)^2},
\end{equation}
\begin{eqnarray}\label{b'2}
\frac{\partial}{\partial y_{i}}b(y_k)&=&-\frac{\lambda p \cdot g(i)}{\left(\sum_{j=0}^{K}y_{j}g(j)\right)^2}\left(y_{k+1}g_{k+1}(k+1)-y_{k}g(k)\right)+i(y_{k}-y_{k-1})\nonumber\\
& &+\lambda \left(1-p+p\frac{g_{k+1}(k+1)}{\sum_{j=0}^{K}y_{j}g(j)}\right)\mathbf{1}_{i=k+1}+\left(\gamma-\sum_{j=0}^{K}jy_{j}\right)\mathbf{1}_{i=k-1}\nonumber\\
& & -\left[\left(1-p+p\frac{g(k)}{\sum_{j=0}^{K}y_{j}g(j)}\right)\lambda  +\gamma-\sum_{j=0}^{K}jy_{j}\right]\mathbf{1}_{i=k}, \quad \text{for } k=1,...,K-1,\nonumber\\
\end{eqnarray}
\begin{equation}\label{b'3}
\frac{\partial}{\partial y_{i}}b(y_K)=\frac{\lambda p y_{K}\cdot g(K)g(i)}{\left(\sum_{j=0}^{K}y_{j}g(j)\right)^2}-iy_{K-1}-\left(1-p+p\frac{g(K)}{\sum_{j=0}^{K}y_{j}g(j)}\right)\lambda\mathbf{1}_{i=K}+\left(\gamma-\sum_{j=0}^{K}jy_{j}\right)\mathbf{1}_{i=K-1}.
\end{equation}

\end{theorem}

\begin{proof}

To prove Theorem \ref{difftheorem}, we take the following 4 steps.

\begin{itemize}
\item[1).](Lemma \ref{martigale_brackets_choice})$\sqrt{N}M_{t}^{N}$ is a family of martingales independent of $D_{0}^{N}$ with Doob-Meyer brackets given by
\begin{eqnarray}
\boldlangle \sqrt{N}M^{N}(k)\boldrangle_{t}&=&\int_{0}^{t}(\beta_{+}(Y_{s}^{N})(k)+\beta_{-}(Y_{s}^{N})(k))ds,\nonumber\\
\boldlangle \sqrt{N}M^{N}(k),\sqrt{N}M^{N}(k+1)\boldrangle_{t}&=&-\int_{0}^{t}\left[\left((1-p)+p\frac{g(k+1)}{\sum_{j=0}^{K}jy_{j}}\right)\lambda Y_{s}^{N}(k+1)\right.\nonumber\\
& & \left.+\left(\frac{M}{N}-\sum_{j=0}^{K}jY_{s}^{N}(j)\right)Y_{s}^{N}(k)\right] ds \quad \text{for } k<K,\nonumber\\
\boldlangle \sqrt{N}M^{N}(k),\sqrt{N}M^{N}(j)\boldrangle_{t}&=&0\quad \text{for } |k-j|>1.
\end{eqnarray}
\item[2).](Lemma \ref{L2bound_choice}) For any $T\geq 0$, $$\limsup_{N\rightarrow \infty}\mathbb{E}(|D_{0}^{N}|^2)<\infty \Rightarrow \limsup_{N\rightarrow \infty}\mathbb{E}(\sup_{0\leq t\leq T}|D_{t}^{N}|^2)<\infty. $$
\item[3).] (Lemma \ref{tightness_choice}) If $(D_{0}^{N})_{N=1}^{\infty}$ is tight then $(D^{N})_{N=1}^{\infty}$ is tight and its limit points are continuous.
\item[4).]	If $D_{0}^{N}$ converges to $D_{0}$ in distribution, then $D_{t}^{N}$ converges to the unique OU process solving $D_{t}=D_{0}+\int_{0}^{t}b'(y_s)D_s ds+M_t$ in distribution.
\end{itemize}

We provide the proofs of Lemma \ref{martigale_brackets_choice}, Lemma \ref{L2bound_choice} and Lemma \ref{tightness_choice} (step 1-3) in the Appendix (Section \ref{App}).   For step 4), by Theorem 4.1 in Chapter 7 of \citet{Ethier2009},  it suffices to prove the following condition holds
$$\sup_{t\leq T}\left|\int_{0}^{t}\left\{\sqrt{N}[b(Y_s^N)-b(y_s)]-b'(y_s)D_{s}^{N}\right\}ds\right|\xrightarrow{p} 0.$$
By Equations (\ref{b'1}), (\ref{b'2}), and (\ref{b'3}), we know that $b(y_t)$ is continuously differentiable with respect to $y_t$. By the mean value theorem, for every $0\leq s\leq t$ there exists a vector $Z_{s}^{N}$ in between $Y_{s}^{N}$ and $y_s$ such that
$$b(Y_{s}^{N})-b(y_s)=b'(Z_{s}^{N})(Y_{s}^{N}-y_s).$$
Therefore
$$\int_{0}^{t}\left\{\sqrt{N}[b(Y_s^N)-b(y_s)]-b'(y_s)D_{s}^{N}\right\}ds=\int_{0}^{t}[b'(Z_{s}^{N})-b'(y_s)]D_{s}^{N}ds.$$
We know that $$\lim_{N\rightarrow \infty}\sup_{t\leq T}|b'(Z_{s}^{N})-b'(y_s)|=0\quad \text{in probability}$$
by the mean field limit convergence and the uniform continuity of $b'$.
By applying Chebyshev's inequality we have   that $D_{s}^{N}$ is bounded in probability. Thus, by Lemma 5.6 in \citet{ko2016strong},
$$\sup_{t\leq T}\left|\int_{0}^{t}\left\{\sqrt{N}[b(Y_s^N)-b(y_s)]-b'(y_s)D_{s}^{N}\right\}ds\right|\xrightarrow{p} 0.$$
\end{proof}


\section{Ratio Process Perspective}\label{Sec_Ratio}

In this section, we explore a new empirical measure process called the ratio process, which is the empirical measure of the ratio between the number of bikes at a station and its capacity, when station capacities differs among stations. According to data from CitiBike, station capacities range from 10 docks to 60 docks. In this case, the ratio process provides a better understanding of the bike availability across the network than the empirical measure process mentioned in previous section, especially for understanding full stations when station capacities differ. Let $K_{\max}$ be the maximum capacity among stations, we define the ratio process $R^N_t$ as follows,
\begin{equation}
R^N_t=\left(R_{t}^{N}(0),R_{t}^{N}(1),\cdots,R_{t}^{N}(K_{\max}-1)\right),
\end{equation}
where for $j=0,\cdots,K_{\max}-1$,
\begin{equation}
R_{t}^{N}(j)=\frac{1}{N}\sum_{i=1}^{N}\mathbf{1}\left\{\frac{j}{K_{\max}} \leq \frac{X_{i}(t)}{K_i}<\frac{j+1}{K_{\max}}\right\}.
\end{equation}

To demonstrate the difference between the empirical measure process and the ratio process in real life setting, we collected CitiBike's real-time system data in the General Bikeshare Feed Specification (GBFS) format. GBFS provides a real-time snapshot of the city's fleet, which includes the number of bikes and docks at every CitiBike station at the time. The data was pulled from CitiBike's API once every second for a period of time in September 2017. Figure \ref{Fig_ratio} gives a screenshot of the empirical measure process of bike distribution and its ratio process at a given time. We can observe that when capacities differ across stations, the ratio process is indeed better at capturing the proportion of stations at full capacity.

\begin{figure}[h]
\centering
\includegraphics[scale=0.5]{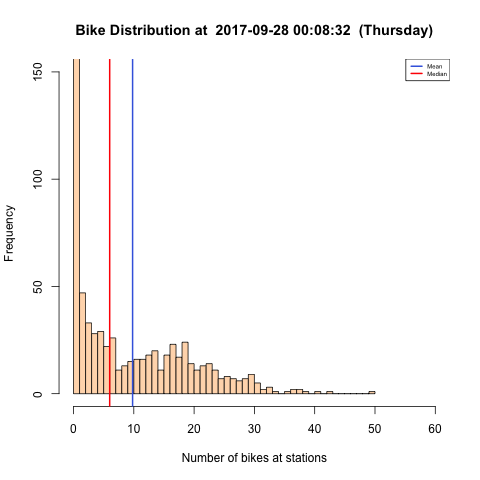}~
\includegraphics[scale=0.5]{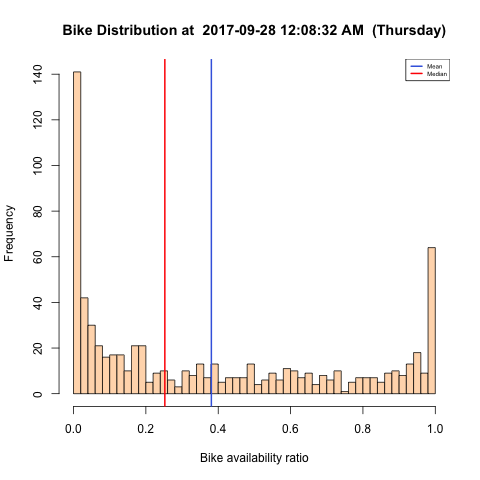}
\caption{A Snapshot of CitiBike Empirical Measure Process vs. Ratio Process}\label{Fig_ratio}
\end{figure}

It is easy to show that $R^N_t$ is a Markov process with the following transition rates:
\begin{itemize}
\item When a customer picks up a bike at station $i$,
\begin{eqnarray}
& &Q^{N}\left(R^N_t,R^N_t+\frac{1}{N}\left(\mathbf{1}_{\floor*{\frac{(X_{i}(t)-1)K_{\max}}{K_{i}}}}-\mathbf{1}_{\floor*{\frac{X_{i}(t)K_{\max}}{K_{i}}}}\right)\right)\nonumber\\
&=&\left((1-p)\lambda_{i}+p\lambda_{i}N\frac{g(X_i(t))}{\sum_{j=1}^{N}g(X_{j}(t))}\right)\mathbf{1}\{ X_{i}(t) > 0 \}.
\end{eqnarray}
\item When a customer drops off a bike at station $i$,
\begin{eqnarray}
& &Q^{N}\left(R^N_t,R^N_t+\frac{1}{N}\left(\mathbf{1}_{\floor*{\frac{(X_{i}(t)+1)K_{\max}}{K_{i}}}}-\mathbf{1}_{\floor*{\frac{X_{i}(t)K_{\max}}{K_{i}}}}\right)\right)\nonumber\\
&=&\mu P_{i}\left(M-\sum_{k=1}^{N}X_{k}(t)\right)\mathbf{1}_{X_{i}(t)<K_i}.
\end{eqnarray}
\end{itemize}

We can prove the following functional forward equations,
 \begin{eqnarray}
& & \lefteqn{ \updot{\mathbb{E}}[f(R^N_t) | R^{N}_0 = r^{N}_0]  \equiv  \frac{d}{dt} \mathbb{E}[f(R^N_t) | R^{N}_0 = r^{N}_0]  } \nonumber \\ 
&=&  \sum_{i=1}^{N}\mathbb{E}\left[ \left( f(R^N_t+\frac{1}{N}\left(\mathbf{1}_{\floor*{\frac{(X_{i}(t)+1)K_{\max}}{K_{i}}}}-\mathbf{1}_{\floor*{\frac{X_{i}(t)K_{\max}}{K_{i}}}}\right)) - f(R^N_t) \right) \cdot \left(\mu P_{i}\left(M\right.\right.\right.\nonumber\\
& -&\left.\left.\left.\sum_{k=1}^{N}X_{k}(t)\right)\mathbf{1}\{X_{i}(t)<K_{i}\}\right) \right]  \nonumber\\
 & +&  \sum_{i=1}^{N}\mathbb{E}\left[ \left( f(R^N_t+\frac{1}{N}\left(\mathbf{1}_{\floor*{\frac{(X_{i}(t)-1)K_{\max}}{K_{i}}}}-\mathbf{1}_{\floor*{\frac{X_{i}(t)K_{\max}}{K_{i}}}}\right)) - f(R^N_t) \right) \cdot \left( \left((1-p)\lambda \right.\right.\right.\nonumber\\
 & +&\left.\left.\left.p\lambda N\frac{g(X_i(t))}{\sum_{j=1}^{N}g(X_{j}(t))}\right)\mathbf{1}\{ X_{i}(t) > 0 \}  \right) \right].
 \end{eqnarray} 
 
The downside of this representation of $R^N_t$ is that the state transition depends on $X_i(t)$ in a very complicated way. Thus, another way to deal with the ratio process is through the empirical measure process $\tilde{Y}_t^{N}(n,k)$, which is the proportion of stations with $n$ bikes and capacity $k$ at time $t$, defined as
 \begin{equation}
\tilde{Y}_t^{N}(n,k)=\frac{1}{N}\sum_{i=1}^{N}\mathbf{1}\{X_{i}(t)=n,K_i=k\}.
 \end{equation}
 Condition on $\tilde{Y}_t^{N}(n,k)=\tilde{y}(n,k)$,  transition rates of $\tilde{y}$ are specified as follows:
\begin{itemize}
\item When a customer arrives to a station with $n$ bikes and capacity $k$ to retrieve a bike,
\begin{eqnarray}
& &Q^{N}\left(\tilde{y},\tilde{y}+\frac{1}{N}(\mathbf{1}_{(n-1,k)}-\mathbf{1}_{(n,k)}) \right)\nonumber\\
&=& \left((1-p)\lambda N\tilde{y}(n,k)+p\lambda N\frac{\tilde{y}(n,k)g(n)}{\sum_{h \in \mathcal{K}}\sum_{j=0}^{h}\tilde{y}(n,h)g(j)}\right) \mathbf{1}_{n>0}.
\end{eqnarray}
\item When a customer returns a bike to a station with $n$ bikes and capacity $k$, 
\begin{eqnarray}
Q^{N} \left(\tilde{y},\tilde{y}+\frac{1}{N}(\mathbf{1}_{(n+1,k)}-\mathbf{1}_{(n,k)}) \right) &=&
 \tilde{y}(n,k)N\left(\frac{M}{N}-\sum_{n'}\sum_{k'}n'\tilde{y}(n',k')\right)\mathbf{1}_{n<k}.\nonumber\\
\end{eqnarray}
\end{itemize}

Define $\mathcal{K}^{N}$ as the set of unique values that station capacity $k$ can take, then we can rewrite the ratio process $R^N_t$ in terms of $\tilde{Y}^N(n,k)$,
\begin{equation}
R_{t}^{N}(j)=\sum_{k\in \mathcal{K}^N}\sum_{n=0}^{k}\tilde{Y}_t^{N}(n,k)\mathbf{1}\left\{\floor*{\frac{nK_{\max}}{k}}=j\right\}.
\end{equation}

The benefit of this representation is that the ratio process $R^N$ is just a linear transformation of these empirical processes $\tilde{Y}_t^N(n,k)$. Thus, we can utilize previous mean field and diffusion limit results we have shown for $\tilde{Y}_t^N(n,k)$ for fixed $k$, to show similar results for the ratio process. In Sections \ref{Fluid_Limit_Ratio} and \ref{Diffusion_Limit_Ratio} we provide mean field limit and diffusion limit analysis for the ratio process.

\subsection{Mean Field Limit for the Ratio Process}\label{Fluid_Limit_Ratio}

Now we state the main theorem in this section that proves the convergence of ratio process to its mean field limit.

\begin{theorem}[Functional Law of Large Number]\label{fluid_limit_ratio}
Let $|. |$ denote the Euclidean norm in $\mathbb{R}^{\mathbb{N}\times \mathcal{K}}$. Suppose that $\lim_{N\rightarrow \infty}\frac{M}{N}=\gamma$ and $\tilde{Y}_{0}^{N}\xrightarrow{p} \tilde{y}_0$. We also assume that capacity $k$ can only take value from a finite set, namely $\mathcal{K}$. Then we have for $\forall \epsilon>0$
$$\lim_{N\rightarrow \infty}P\left(\sup_{t\leq t_0}|R_t^N-r_t|>\epsilon\right)=0$$
where $$r_t(j)=\sum_{k\in \mathcal{K}}\sum_{n=0}^{k}\tilde{y}_t(n,k)\mathbf{1}\left\{\floor*{\frac{nK_{\max}}{k}}=j\right\}$$
and  $\tilde{y}_t$ is the unique solution to the following differential equation starting at $\tilde{y}_0$
\begin{equation}\label{diff_eqn_ratio}
\shortdot{\tilde{y}_{t}}=b(\tilde{y}_{t})
\end{equation}

where $b:[0,1]^{\mathbb{N}\times \mathcal{K}}\rightarrow \mathbb{R}^{\mathbb{N}\times \mathcal{K}}$ is a vector field satisfies
\begin{eqnarray}\label{eqn:b_ratio}
b(\tilde{y}_{t})&=&\sum_{k\in \mathcal{K}}\sum_{n=0}^{k}\left[ \left((1-p)+p\frac{g(n)}{\sum_{j=0}^{k}\tilde{y}(j,k)g(j)}\right)(\mathbf{1}_{n-1,k}-\mathbf{1}_{n,k})\lambda\mathbf{1}_{n>0}\right.\nonumber\\
& &\left.+\left(\gamma-\sum_{k\in \mathcal{K}}\sum_{j=0}^{k} j\tilde{y}(j,k) \right)(\mathbf{1}_{n+1,k}-\mathbf{1}_{n,k})\mathbf{1}_{n<k}\right] y_{t}(n,k) 
\end{eqnarray}
or for any $k\in \mathcal{K}$,
\begin{equation}
b(\tilde{y}(0,k))=\underbrace{-\left(\gamma-\sum_{h\in \mathcal{K}}\sum_{j=0}^{h} j\tilde{y}(j,h) \right)\tilde{y}(0,k)}_{\text{return a bike to a no-bike station}}+\underbrace{\left((1-p)+p\frac{g(1)}{\sum_{h \in \mathcal{K}}\sum_{j=0}^{K}\tilde{y}(j,h)g(j)}\right)\lambda \tilde{y}(1,k)}_{\text{retrieve a bike from a 1-bike station}},
\end{equation}

\begin{equation}
\begin{split}
b(\tilde{y}(i,k))=& \underbrace{\left((1-p)+p\frac{g(i)}{\sum_{h \in \mathcal{K}}\sum_{j=0}^{K}\tilde{y}(j,h)g(j)}\right)\lambda \tilde{y}(i+1,k)}_{\text{retrieve a bike from a $i+1$-bike station}}+\underbrace{\left(\gamma-\sum_{h\in \mathcal{K}}\sum_{j=0}^{h} j\tilde{y}(j,h) \right)y(i-1,k)}_{\text{return a bike to a $i-1$-bike station}}\\
& -\underbrace{\left(\left((1-p)+p\frac{g(i)}{\sum_{h \in \mathcal{K}}\sum_{j=0}^{K}\tilde{y}(j,h)g(j)}\right)\lambda+\gamma-\sum_{h\in \mathcal{K}}\sum_{j=0}^{h} j\tilde{y}(j,h) \right)\tilde{y}(i,k)}_{\text{retrieve and return a bike to a $i$-bike station}},
\end{split}
\end{equation}

for $ i=1,...,k-1$ and 
\begin{equation}
b(\tilde{y}(k,k))=\underbrace{-\left((1-p)+p\frac{g(k)}{\sum_{h \in \mathcal{K}}\sum_{j=0}^{K}\tilde{y}(j,h)g(j))}\right)\lambda \tilde{y}(k,k)}_{\text{retrieve a bike from a $k$-bike station}}+\underbrace{\left(\gamma-\sum_{h\in \mathcal{K}}\sum_{j=0}^{h} j\tilde{y}(j,h) \right)\tilde{y}(k-1,k)}_{\text{return a bike to a $k-1$-bike station}}.
\end{equation}
\end{theorem}
\begin{proof}
Since
\begin{eqnarray}
|R_t^N-r_t|^2 &\leq & \sum_{j=0}^{K_{max}}\left(\sum_{k\in \mathcal{K}}\sum_{n=0}^{k}(\tilde{Y}^{N}_t(n,k)-\tilde{y}_t(n,k))\mathbf{1}\left\{\floor*{\frac{nK_{\max}}{k}}=j\right\}\right)^2 \nonumber\\
&\leq & K_{max}(K_{max}+1)n(\mathcal{K}) \sum_{k \in \mathcal{K}}\sum_{n=0}^{k}|\tilde{Y}^{N}_t(n,k)-\tilde{y}_t(n,k)|^2 \nonumber \\
&\leq & 2K_{max}^3\sum_{k \in \mathcal{K}}|\tilde{Y}^{N}_t(\cdot,k)-\tilde{y}_t(\cdot,k)|^2,
\end{eqnarray}
for any $k \in \mathcal{K}$, by Theorem \ref{fluid_limit}, we know that
$$\lim_{N\rightarrow \infty}P\left(\sup_{t\leq t_0}|\tilde{Y}^N_t(\cdot,k)-\tilde{y}_t(\cdot,k)|>\epsilon\right)=0.$$
Thus, since $\mathcal{K}$ is a finite set, it implies
$$\lim_{N\rightarrow \infty}P\left(\sup_{t\leq t_0}|R^N_t-r_t|>\epsilon\right)\leq \lim_{N\rightarrow \infty}P\left(\sup_{t\leq t_0}\sum_{k \in\mathcal{K}}|\tilde{Y}^N_t(\cdot, k)-\tilde{y}_t(\cdot, k)|^2>\epsilon^2/2K_{max}^3\right)=0.$$
\end{proof}


\subsection{Diffusion Limit for the Ratio Process}\label{Diffusion_Limit_Ratio}
We define the diffusion scaled ratio process by substracting the mean field limit from the ratio process and rescaling by $\sqrt{N}$.  Thus, we obtain the following expression for the diffusion scaled ratio process
\begin{eqnarray}
Z_{t}^{N}&=&\sqrt{N}(R_{t}^{N}-r_{t}).
\end{eqnarray}
For each $j=0,\cdots,K_{max}$, we have
\begin{eqnarray}
Z_{t}^{N}(j)&=&\sum_{k\in \mathcal{K}}\sum_{n=0}^{k}\sqrt{N}(\tilde{Y}^{N}_t(n,k)-\tilde{y}_t(n,k))\mathbf{1}\left\{\floor*{\frac{nK_{\max}}{k}}=j\right\}\nonumber\\
&=& \sum_{k\in \mathcal{K}}\sum_{n=0}^{k}D^{N}_{t}(n,k)\mathbf{1}\left\{\floor*{\frac{nK_{\max}}{k}}=j\right\}.
\end{eqnarray}

Now we state the functional central limit theorem for the ratio process as follows.
\begin{theorem}\label{difftheorem_ratio}
Consider $Z_{t}^{N}$ in $\mathbb{D}(\mathbb{R}_{+},\mathbb{R}^{K+1})$ with the Skorokhod $J_{1}$ topology, and suppose that 
$\limsup_{N\rightarrow \infty}\sqrt{N}(\frac{M}{N}-\gamma)< \infty$. 
Then if $Z_{0}^{N}$ converges in distribution to $Z_{0}$, then $Z_{t}^{N}$ converges to the unique OU process solving $Z_{t}=Z_{0}+\int_{0}^{t}b'(y_s)Z_s ds+\tilde{M}_t$ in distribution, where 
$$\tilde{M}_t(j)=\sum_{k\in \mathcal{K}}\sum_{n=0}^{k}M_{t}(n,k)\mathbf{1}\left\{\floor*{\frac{nK_{\max}}{k}}=j    \right\}$$
for $j=0,\cdots,K_{\max}$. 
\end{theorem}

\begin{proof}
We have shown in Theorem \ref{difftheorem} that $$D_{t}^{N}(n,k)\Rightarrow D_{t}(n,k)$$
for any $k\in \mathcal{K}$ and $0\leq n\leq k$. 
We also know that
$$\floor*{\frac{nK_{\max}}{k}}=j \Rightarrow j\leq \frac{nK_{\max}}{k}<j+1 \Rightarrow j\cdot \frac{k}{K_{\max}}\leq n< (j+1)\cdot \frac{k}{K_{\max}}.$$
Since $k/K_{\max}<1$, there exists at most one  $n^*(j,k)$ in $\{0,\cdots,k\}$ such that $\floor*{\frac{n^*(j,k)K_{\max}}{k}}=j$. Then
\begin{eqnarray}
Z_{t}^{N}(j)&=& \sum_{k\in \mathcal{K}}D^{N}_{t}(n^*(j,k),k).
\end{eqnarray}
Since all $k$'s in $\mathcal{K}$ are distinct, $\{D^{N}_{t}(n^*(j,k),k)\}_{k\in \mathcal{K}}$ are independent. This can be shown by the fact that $Y^{N}(n,k)$ and $Y^{N}(n',k')$ are independent when $k\neq k'$. Thus, 
\begin{eqnarray}
\sum_{k\in \mathcal{K}}D^{N}_{t}(n^*(j,k),k)\Rightarrow \sum_{k\in \mathcal{K}}D_{t}(n^*(j,k),k).
\end{eqnarray}
Define $$Z_{t}(j)=\sum_{k\in \mathcal{K}}D_{t}(n^*(j,k),k),$$
then we proved that $Z_{t}^{N}$ converge in distribution to $Z_{t}$, and that $Z_t$ satisfies the SDE $Z_{t}=Z_{0}+\int_{0}^{t}b'(y_s)Z_s ds+\tilde{M}_t$.
\end{proof}

\section{Steady State Analysis}\label{Sec_SteadyState}

In this section, we provide the steady state analysis of the mean field limit of both the empirical measure process and the ratio process. First, we compute explicitly the equilibrium point and the Lyapunov functions for the mean field limit. We also show the interchanging of limits result for both the empirical measure process and the ratio process, that is, the stationary distribution $Y^N(\infty)$ of the empirical process converges to the same equilibrium point $\bar{y}$ of the mean field limit. 

\subsection{Equilibrium Analysis for Empirical Process}\label{equil_Y}
Our goal in this section is to prove the following diagram commutes,
\begin{displaymath}
    \xymatrix{
        Y^{N}(t)\ar[r]^{ N\rightarrow \infty} \ar[d]_{t\rightarrow \infty} & y(t) \ar[d]^{t \rightarrow \infty} \\
        Y^{N}(\infty) \ar[r]_{N\rightarrow \infty}       & y(\infty) }
\end{displaymath}

We have showed in Section \ref{Fluid_Limit} that $Y^{N}(t)\xrightarrow{p} y(t)$, we will now give the proof of the existence and uniqueness of $y(\infty)$, as well as the convergence of invariant measure $Y^{N}(\infty)$ to $y(\infty)$.

\begin{proposition}\label{unique}
The equilibrium point $\bar{y}$ for the mean field limit exists and is unique, and it satisfies the following equations
\begin{equation}
\bar{y}_n=\frac{\prod_{i=0}^{n-1}\rho_i(\bar{y})}{1+\sum_{i=0}^{K}\prod_{k=0}^{i-1}\rho_k(\bar{y})}
\end{equation}
for $n=0,\cdots,K$ where
\begin{equation}\label{rho}
\rho_k(\bar{y})=\frac{\gamma-\sum_{j=0}^{K}j\bar{y}_j}{\lambda\left(1-p+p\frac{g(k)}{\sum_{j=0}^{K}g(j)\bar{y}_j}\right)}
\end{equation}
for $k=0,\cdots,K-1$.
\end{proposition}

\begin{proof}

Denote $\bar{y}$ as the equilibrium point for the mean field limit, then this queue can be modeled as a standard birth-death process and
analyzed using the basic relation
$$\left(\gamma-\sum_{j=0}^{K}j\bar{y}_j\right) \bar{y}_k=\lambda\left(1-p+p\frac{g(k)}{\sum_{j-0}^{K}g(j)\bar{y}_j}\right) \bar{y}_{k+1}.$$
Denote \begin{equation}\label{rho}
\rho_k(\bar{y})=\frac{\gamma-\sum_{j=0}^{K}j\bar{y}_j}{\lambda\left(1-p+p\frac{g(k)}{\sum_{j=0}^{K}g(j)\bar{y}_j}\right)}
\end{equation}
for $k=0,\cdots,K-1$.
Then from the recurring relationship we get 
$$\bar{y}_k=\bar{y}_0 \prod_{i=0}^{k-1}\rho_i(\bar{y}).$$
Since $\sum_{k=0}^{K}\bar{y}_k=1$, we have 
$$\bar{y}_0=\frac{1}{1+\sum_{i=1}^{K}\prod_{k=0}^{i-1}\rho_k(\bar{y})}.$$
Then $\bar{y}$ satisfies the following equations,
\begin{equation}
\bar{y}_n=\frac{\prod_{i=0}^{n-1}\rho_i(\bar{y})}{1+\sum_{i=1}^{K}\prod_{k=0}^{i-1}\rho_k(\bar{y})}
\end{equation}
for $n=0,\cdots,K$.

To show the existence and uniqueness of $\bar{y}$, we rewrite Equation (\ref{rho}) as 
\begin{equation}\label{root}
\gamma=\lambda\rho_k(\bar{y})\left(1-p+p\frac{g(k)}{\sum_{j=0}^{K}g(j)\bar{y}_j}\right)+\sum_{j=0}^{K}j\bar{y}_j\triangleq f(\rho_k).
\end{equation}
For any $0\leq k \leq K-1$, denote the right-hand side of the equation to be a function of $\rho_k$, then 
\begin{equation}
f'(\rho_k)=\lambda (1-p)+\lambda p g(k)\frac{\sum_{j=0}^{K}g(j)\bar{y}_j-\rho_k \sum_{j=0}^{K}g(j)\bar{y}_j'(\rho_k)}{\left(\sum_{j=0}^{K}g(j)\bar{y}_j\right)^2}+\sum_{j=0}^{K}j\bar{y}_j'(\rho_k).
\end{equation}
Since
\begin{eqnarray}
\bar{y}_j-\rho_k\bar{y}_j'(\rho_k)&=&\frac{\prod_{i=0}^{j-1}\rho_i}{1+\sum_{i=0}^{K-1}\prod_{h=0}^{i}\rho_h}-\frac{\prod_{i=0}^{j-1}\rho_i \left( \mathbf{1}_{j > k}-\sum_{i=k}^{K-1}\prod_{h=0}^{i}\rho_h\right)}{\left(1+\sum_{i=0}^{K-1}\prod_{h=0}^{i}\rho_h\right)^2}\nonumber\\
&=&\frac{\prod_{i=0}^{j-1}\rho_i}{\left(1+\sum_{i=0}^{K-1}\prod_{h=0}^{i}\rho_h\right)^2}\left(1-\mathbf{1}_{j> k}+\sum_{i=k}^{K-1}\prod_{h=0}^{i}\rho_h\right)\nonumber\\
&\geq &0,
\end{eqnarray}
the second term of $f'(\rho_k)$ is positive. We also have that
\begin{equation*}
\bar{y}_j'(\rho_k)<0 \text{ when } j\leq k 
\end{equation*}
\begin{equation*}
\bar{y}_j'(\rho_k)>0 \text{ when } j > k 
\end{equation*}
which means the third term $\sum_{j=0}^{K}j\bar{y}_j'(\rho_k)$ is also positive. Thus, $f(\rho_k)$ is an increasing function of $\rho_k$ on $\mathbb{R}^+$ and Equation (\ref{root}) has a unique root $\rho_k>0$. Therefore, we have shown that $\bar{y}$ exists and is unique.
\end{proof}


Now we proceed to showing a Lyapunov function of the mean field limit $y(t)$. To simplify the notation, we denote the distribution $\nu_{\rho(y)}$ as
$$\nu_{\rho(y)}(n)=\frac{\prod_{i=0}^{n-1}\rho_i(y)}{1+\sum_{i=1}^{K}\prod_{k=0}^{i-1}\rho_k(y)}$$
for $n=0,\cdots,K$ where 
$$\rho_k(y)=\frac{\gamma-\sum_{j=0}^{K}jy_j}{\lambda\left(1-p+p\frac{g(k)}{\sum_{j=0}^{K}g(j)y_j}\right)}$$
for $k=0,\cdots,K-1$. We also define 
$$Z(\rho)=1+\sum_{i=1}^{K}\prod_{k=0}^{i-1}\rho_k$$
and $$\phi_{k,n}(y_n)=\rho_k(y)$$
Then we have the following theorem of the Lyapunov function for $y(t)$.
\begin{theorem}[Lyapunov function] \label{Lyapunov}
Denote $\mathcal{X}=\{0,1,\cdots,K\}$ and $\mathcal{P}(\mathcal{X})$ as the set of all probability measures on $\mathcal{X}$. Then for any $y\in \mathcal{P}(\mathcal{X})^{\mathrm{o}}$, define $g(y)=h(y)-\log(Z(\rho(y))+\sum_{n=0}^{K}\sum_{k=0}^{n-1}S_{k,n}(y_n))$ where
$$h(y)=\sum_{n=0}^{K}y_n\log(y_n/\nu_{\rho(y)}(n))$$
is the relative entropy of $y$ w.r.t $\nu_{\rho(y)}$,
and
$S_{k,n}$ is some primitive of $x\mapsto x\frac{\phi_{k,n}'(x)}{\phi_{k,n}(x)}$. Then $g$ is a Lyapunov function for the dynamical system $y(t)$.


\end{theorem}
\begin{proof}
The proof is provided in the Appendix (Section \ref{App}).
\end{proof}

Now we state the main theorem of this section, the interchanging of limits result for the mean field limit.

\begin{theorem}\label{converge}
The sequence of invariant probability distributions $\bar{y}^N$ of process $Y^{N}_t$ converges weakly to he Dirac mass at $\bar{y}$ as $N\rightarrow \infty$.
\end{theorem}

\begin{proof}
The Lyapunov function we introduced in Theorem \ref{Lyapunov} plays a key role in the proof of this result.  We will be showing the result in the following steps:
\begin{itemize}
\item[1)]The infinitesimal generator $Q_{N}$ of the empirical measure process converges to the  degenerate generator $Q_{\infty}(\cdot)(y)$, which satisfies $Q_{\infty}(F)(y)=\left<\nabla F(y),b(y)\right>$ for any $y\in \mathcal{P}(\mathcal{X})$ and $F\in C^{2}(\mathbb{R}^{K+1})$.
\item[2)] For any subsequence $\bar{y}^{N_p}$ of the invariant measures $\bar{y}^{N}$, its limit $\tilde{y}$ solves $Q_{\infty}\tilde{y}=0$, which means it is the invariant measure of the markov chain with infinitesimal generator $Q_{\infty}$.
\item[3)] Using the Lyapunov function $g$, we can show that $Q_{\infty}(g)\tilde{y}=\int_{\mathcal{P}(\mathcal{X})}yB_y\nabla g(y) \tilde{y}(dy)=0$. The integral being non-positive, and zero only at $\bar{y}$ , means $\tilde{y}$ is also the equilibrium of the dynamical system.
\item[4)] By uniqueness of the equilibrium point and compactness of $\mathcal{P}(\mathcal{X})$, $\bar{y}$ is the only weak limit of $\bar{y}^{N}$.
\end{itemize}

Now define the functional operator associated with the Q-matrix as 
\begin{eqnarray}
Q_{N}(F)(y)&=&\sum_{z\neq y}Q_N(y,z)(F(y)-F(z))\nonumber\\
&=& \sum_{n=0}^{K}\left[ \left((1-p)+p\frac{g(n)}{\sum_{j=0}^{K}y_{j}g(j)}\right)\left(F\left(y+\frac{1}{N}(\mathbf{1}_{n-1}-\mathbf{1}_{n})\right)-F(y)\right)\lambda N\mathbf{1}_{n>0}\right.\nonumber\\
& &\left.+N\left(\frac{M}{N}-\sum_{j=0}^{K} jy_{j} \right)\left(F\left(y+\frac{1}{N}(\mathbf{1}_{n+1}-\mathbf{1}_{n})\right)-F(y)\right)\mathbf{1}_{n<K}\right] y_n.
\end{eqnarray}
for any $F\in C^2(\mathbb{R}^{K+1})$. Now, it is easy to check that the sequence $Q_{N}(F)(y)$ converges to the following
\begin{eqnarray}
& &\sum_{n=0}^{K}\left[ \left((1-p)+p\frac{g(n)}{\sum_{j=0}^{K}y_{j}g(j)}\right)\left< \nabla F(y),\mathbf{1}_{n-1}-\mathbf{1}_{n}\right> \lambda \mathbf{1}_{n>0}\right.\nonumber\\
& &\left.+\left(\gamma-\sum_{j=0}^{K} jy_{j} \right)\left< \nabla F(y),\mathbf{1}_{n-1}-\mathbf{1}_{n}\right>\mathbf{1}_{n<K}\right] y_n.
\end{eqnarray}
which we define as $Q_{\infty}(F)(y)$. Moreover, by using Taylor’s formula at the second order, this convergence is uniform with respect to $y\in \mathcal{P}(\mathcal{X})$. Thus we have
$$Q_{\infty}(F)(y)=\left<\nabla F(y),b(y)\right>, y\in \mathcal{P}(\mathcal{X}).$$
Now since $\mathcal{X}$ is a finite set, the set of $\mathcal{P}(\mathcal{X})$ is compact, and the sequence of invariant measure $\bar{y}^N$ is relatively compact.  Let $\tilde{y}$ be the limit of some subsequence $\bar{y}^{N_p}$. Then for any $F\in C^2(\mathbb{R}^{K+1})$ and $p\geq 0$,
$$\int_{\mathcal{P}(\mathcal{X})}Q_{N_p}(F)(y)\bar{y}^{N_p}(dy)=0.$$
The uniform convergence of $Q_{N_p}(F)$ to $Q_{\infty}(F)$ implies
$$0=\int_{\mathcal{P}(\mathcal{X})}Q_{\infty}(F)(y)\tilde{y}(dy)=\int_{\mathcal{P}(\mathcal{X})}\left<\nabla F(y),b(y)\right>\tilde{y}(dy).$$
Then $\tilde{y}$ is an invariant measure of the Markov process associated with the infinitesimal generator $Q_{\infty}$.

Now we show that $y(t)$ must lies in the interior of $\mathcal{P}(\mathcal{X})$. For $t\geq 0$, denote $y(x, t)$ as the dynamical system at  time $t$ starting from $y_0=x\in \mathcal{P}(\mathcal{X})$. Assume that there exists $x\in \mathcal{P}(\mathcal{X})$, $s>0$ and $n\in \mathcal{X}$ such that $y_n(x,s)=0$. Since
$y_n(x, t)$ is non-negative and the function $t\rightarrow y(x, t)$ is of class $C^1$, it implies that $b_n(y(x,s))=\shortdot{y}_n(x,s)=0$. By the definition of $b_n(y)$  and that $ y_n(x,s)=0$, we know that $y_{n+1}(x,s)=y_{n-1}(x,s)=0$ when $n+1$ or $n-1$ is in $\mathcal{X}$. By repeating the same argument, we can conclude that all coordinates of $y(x,s)$ are zeros. This contradicts the fact that $y(x,s)$ is a probability measure. Hence the boundary $\partial \mathcal{P}(\mathcal{X})$ can't be reached by $y(x,t)$ in positive time. 

For $x\in \mathcal{P}(\mathcal{X})$ and $0<s'<s$, since $g(y(x,\cdot))$ is in class $C^1$ on $[s',s]$ and its derivative is $\left<\nabla g(y(x,\cdot)), b(y(x,\cdot))\right>$, we have
$$g(y(x,s))-g(y(x,s'))=\int_{s'}^{s}Q_{\infty}(g)(y(x,u))du.$$
By integrating both sides w.r.t $\tilde{y}$, the invariance of $\tilde{y}$ for the
process $y(x, t)$ and Fubini’s theorem show that
\begin{eqnarray}
0&=&\int_{\mathcal{P}(\mathcal{X})}g(y(x,s))\tilde{y}(dx)-\int_{\mathcal{P}(\mathcal{X})}g(y(x,s'))\tilde{y}(dx)\nonumber\\
&=&\int_{\mathcal{P}(\mathcal{X})}\int_{s'}^{s}Q_{\infty}(g)(y(x,u))du\tilde{y}(dx)\nonumber\\
&=&(s-s')\int_{\mathcal{P}(\mathcal{X})}Q_{\infty}(g)(x)\tilde{y}(dx).
\end{eqnarray}
By Theorem \ref{Lyapunov}, $\left<\nabla g(x), b(x)\right>\leq 0$ for all $x$, thus $Q_{\infty}(g)(x)=0$ $\tilde{y}$-almost surely. Thus $\tilde{y}$ is also the equilibrium of the dynamical system. Since we showed in Proposition \ref{unique} that the equilibrium point $\bar{y}$ is unique, we can conclude that the sequence of invariant probability measure $\bar{y}^N$ of process $Y^{N}(t)$ converges weakly to he Dirac mass at $\bar{y}$ as $N\rightarrow \infty$.

\end{proof}

\subsection{Equilibrium Analysis for Ratio Process}\label{equil_R}
In this section, we provide a similar equilibrium analysis for the ratio process $R^N_t$, which utilize the results we obtained in Section \ref{equil_Y}.
\begin{proposition}
The equilibrium point $\bar{r}$ for the mean field limit of the ratio process $r(t)$ exists and is unique, and it satisfies the following equation
$$\bar{r}=\sum_{k\in \mathcal{K}}y_{K}(k)\bar{y}(\cdot|K=k)A^{(k)}$$
where for each $k\in \mathcal{K}$, $A^{(k)}$ is a $(k+1)\times (K_{\max}+1)$ matrix defined as 
$$A^{(k)}_{ij}=\mathbf{1}\left\{\frac{iK_{\max}}{k}=j \right\}$$ for $i=0,\cdots,k$ and $j=0,\cdots,K_{\max}$ and  
$$y_{K}(k)=\sum_{n=0}^{k}\tilde{y}(n,k)$$ is the marginal distribution for capacity $k$.
\end{proposition}

\begin{proof}

By the definition of the ratio process in Section \ref{Sec_Ratio}, the ratio process can be written as
$$R^N_t=\sum_{k\in \mathcal{K}^N}\sum_{n=0}^{k}\tilde{Y}_t^{N}(n,k)\mathbf{1}\left\{\floor*{\frac{nK_{\max}}{k}}=j\right\}=\sum_{k\in \mathcal{K}^{N}}\tilde{Y}_{t}^{N}(\cdot,k)A^{(k)}.$$
Similarly, the mean field limit for the ratio process $R^N_t$ can be written as
$$r_t=\sum_{k\in \mathcal{K}}\tilde{y}_{t}(\cdot,k)A^{(k)}.$$
Let the marginal distribution for capacity $k$ be 
$y_{K}(k)=\sum_{n=0}^{k}\tilde{y}(n,k)$, then the network of all stations with capacity $k$ can be viewed as a closed queueing network and its empirical mean field limit $y_t(\cdot |K=k)$ can be written as
$$y_t(n |K=k)=\frac{\tilde{y}_t(n,k)}{y_{K}(k)}$$
for $n=0,\cdots,k$. 
We have shown the uniqueness and existence of equilibrium points for queueing networks with same capacities, therefore the equilibrium points $\bar{r}$ for the ratio process also exists and is unique, and it satisfies
$$\bar{r}=\sum_{k\in \mathcal{K}}y_{K}(k)\bar{y}(\cdot|K=k)A^{(k)}.$$
\end{proof}

\begin{theorem}
The sequence of invariant probability distributions $\bar{r}^N$ of the ratio process $R^N_t$ converges weakly to the Dirac mass at $\bar{r}$ as $N\rightarrow \infty$.
\end{theorem}
 
\begin{proof}
By Theorem \ref{converge}, the invariant probability distribution $\bar{y}^N$ of process $Y^{N}(t)$ converges weakly to the Dirac mass at $\bar{y}$ as $N\rightarrow \infty$. We know that the invariant probability distribution $\bar{r}^N$ of ratio process $R^N_t$ can be expressed as 
$$\bar{r}^N(j)=\sum_{k\in \mathcal{K}}\sum_{n=0}^{k}\bar{y}^{N}(n,k)\mathbf{1}\left\{\floor*{\frac{nK_{\max}}{k}}=j\right\},$$
and that $$\bar{r}(j)=\sum_{k\in \mathcal{K}}\sum_{n=0}^{k}\bar{y}(n,k)\mathbf{1}\left\{\floor*{\frac{nK_{\max}}{k}}=j\right\}$$
for $j=0,\cdots,K_{\max}$.

We also know that
$$\floor*{\frac{nK_{\max}}{k}}=j \Rightarrow j\leq \frac{nK_{\max}}{k}<j+1 \Rightarrow j\cdot \frac{k}{K_{\max}}\leq n< (j+1)\cdot \frac{k}{K_{\max}}.$$
Since $k/K_{\max}<1$, there exists at most one  $n^*(j,k)$ in $\{0,\cdots,k\}$ such that $\floor*{\frac{n^*(j,k)K_{\max}}{k}}=j$. Then
\begin{eqnarray}
\bar{r}^{N}(j)&=& \sum_{k\in \mathcal{K}}\bar{y}^{N}(n^*(j,k),k).
\end{eqnarray}
Since convergence in distribution to a constant implies convergence in probability, and by additivity of convergence in probability, we have
\begin{eqnarray}
\sum_{k\in \mathcal{K}}\bar{y}^{N}(n^*(j,k),k)\xrightarrow{p} \sum_{k\in \mathcal{K}}\bar{y}(n^*(j,k),k),
\end{eqnarray}
which implies
\begin{eqnarray}
\sum_{k\in \mathcal{K}}\bar{y}^{N}(n^*(j,k),k)\Rightarrow \sum_{k\in \mathcal{K}}\bar{y}(n^*(j,k),k).
\end{eqnarray}
This shows that the invariant measure $\bar{r}^N$ converges weakly to the Dirac mass at $\bar{r}$ as $N\rightarrow \infty$.
\end{proof}

\section{Insights from Numerical Examples and Simulation}\label{Sec_simulation}
In this section, we show the different aspects of the impact of information on the empirical measure. This includes how the proportion of users that have information (parameter $p$), different choice function $g$ (exponential, minimum choice functions) and users' sensitivity to information (parameters in $g$) affect the behavior of empirical measure over time. We divide the section into two subsections, where Section \ref{Sec_stationary} considers empirical measure under stationary arrivals, and Section \ref{Sec_nonstationary} considers empirical measure under non-stationary arrivals.

%
%
%
%
%
%
%


\subsection{Stationary Arrivals}\label{Sec_stationary}
In this section, we consider the bike sharing system under stationary arrivals. Common parameter values used in all numerical examples are $\lambda=1, \mu=1, K=20, \gamma=10$. We explore the impact of information on three different perspectives: parameter $p$, different choice function $g$, and parameters in each choice function $g$.
\subsubsection{Impact of proportion of users with information ($p$) to the empirical measure}
We first explore impact of parameter $p$, the proportion of user with information, on the empirical measure. Here we consider three different choice function: 1) exponential choice function $g(x)=e^{\theta x}$; 2) minimum choice function: $g(x)=\min(x,c)$; 3) polynomial choice function: $g(x)=x^{\alpha}$.

Another aspect we considered is the entropy of the empirical process. We define the \textbf{entropy} of the empirical measure as 
\begin{equation}
S^N(t)=-\sum_{i=0}^{K}Y^N_t(i)\log(Y^N_t(i)).
\end{equation}
In our model, our goal is to minimize the proportion of empty and full stations, thus we want the empirical measure  to be close to a mass at the average number of bikes per stations, that is, low entropy.

\paragraph{Exponential Choice Function $g(x)=e^{\theta x}$} 
We first explore the exponential choice function. This represent a high user sensitivity to information about bike availability, as the probability of choosing stations with more bikes grows exponentially.

Figure \ref{Fig_bike_dist_exp10} shows the empirical process at equilibrium for $p=0, 0.25, 0.5, 1$ and $\theta=1$. We observe that with no information ($p=0$) the distribution looks relatively flat. As you increase the value of $p$, i.e. the proportion of users with information, the distribution narrows down to be centered around its mean  quite  drastically. We also observe that the proportion of empty and full stations go close to 0 with only 25\% of user having information.

Figure \ref{Fig_y_p_theta} shows the surface plots of the mean field limit on the $p-\theta$ plane with $p\in [0,1]$ and $\theta\in [0,2]$. Particularly, we look at proportion of stations with 0,1,19 and 20 bikes. These plots show that proportion of nearly empty and full stations go down dramatically both when $p$ increase from 0 to 1 and $\theta$ increase from 0 to 2.

Figure \ref{Fig_entropy_p_theta} shows the surface plot of the entropy of the mean field limit on the $p-\theta$ plane. We observe that the entropy goes down as $p, \theta$ increase, which is consistent with the findings in Figures \ref{Fig_y_p_theta} and \ref{Fig_entropy_p_theta}.

\begin{figure}[h]
\centering
\includegraphics[width=0.75\textwidth]{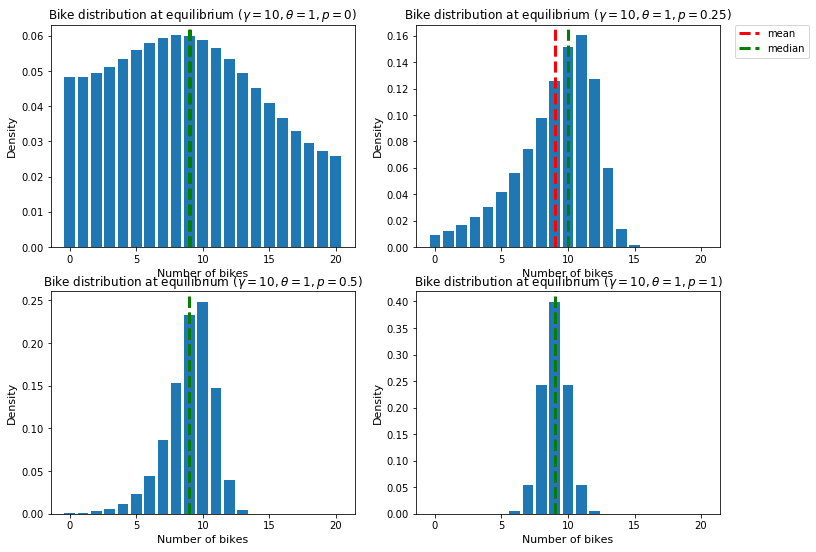}
\caption{Bike distribution for different values of $p$ when $\theta=2$ }\label{Fig_bike_dist_exp10}
\end{figure}

\begin{figure}[H]
\centering
\includegraphics[width=1\textwidth]{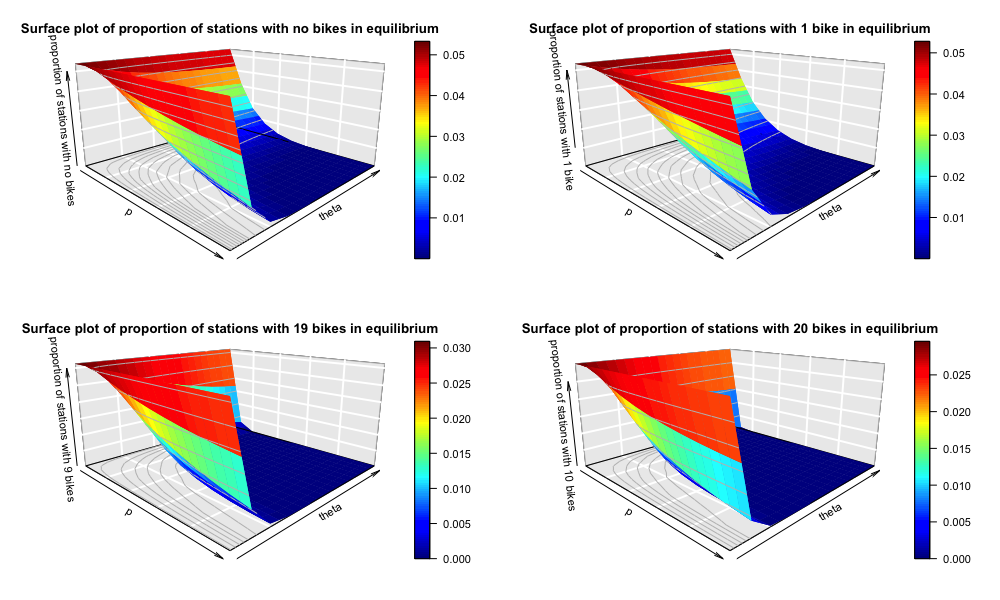}
\caption{Surface plot of $y_{\infty}(p,\theta)$ on the $p-\theta$ plane}\label{Fig_y_p_theta}
\end{figure}

\begin{figure}[H]
\centering
\includegraphics[width=0.8\textwidth]{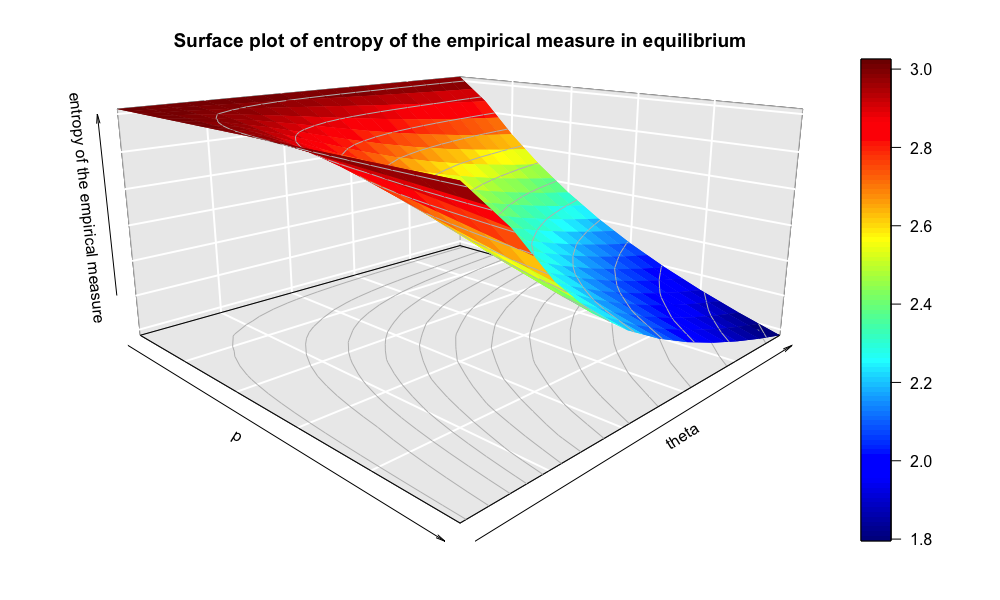}
\caption{Surface plot of the entropy of the empirical measure in equilibrium on the $p-\theta$ plane}\label{Fig_entropy_p_theta}
\end{figure}

\paragraph{Minimum Choice Function $g(x)=\min(x,c)$} 
We then explore the minimum choice function. This choice function represents that user sensitivity to information about bike availability has a threshold, as the probability of choosing stations with more bikes grows linearly but then remain the same once the threshold $c$ is reached.

Figure \ref{Fig_bike_dist_c5} shows the empirical process at equilibrium for $p=0, 0.25, 0.5, 1$ and $c=5$. We observe that with no information ($p=0$) the distribution looks relatively flat. As the value of $p$ increase to 1, i.e. the proportion of users with information, the proportion of stations with empty stations goes down to close to 0. However, we don't see much changes in terms of the proportion of full or nearly full stations. This agrees with the intuition that when there is  a threshold on the incentives the number of bikes available has on users, its effect on reducing full stations is diminished.

Figure \ref{Fig_y_p_c} shows the plots of the mean field limit vs. proportion of users with information ($p$) for different values of $c$. Particularly, we look at proportion of stations with 0,1,19 and 20 bikes. These plots show that proportion of empty stations go down both when $p$ increase from 0 to 1 and $c$ increase from 1 to 5. However, the proportion of full stations doesn't change much for different values of $p$ and $c$.

Figure \ref{Fig_entropy_p_c} shows the plot of the entropy of the mean field limit vs. proportion of users with information ($p$) for different values of $c$. We observe that the entropy goes down as $p, c$ increase, which is consistent with the findings in Figures \ref{Fig_y_p_c} and \ref{Fig_entropy_p_c}.

\begin{figure}[H]
\centering
\includegraphics[width=0.75\textwidth]{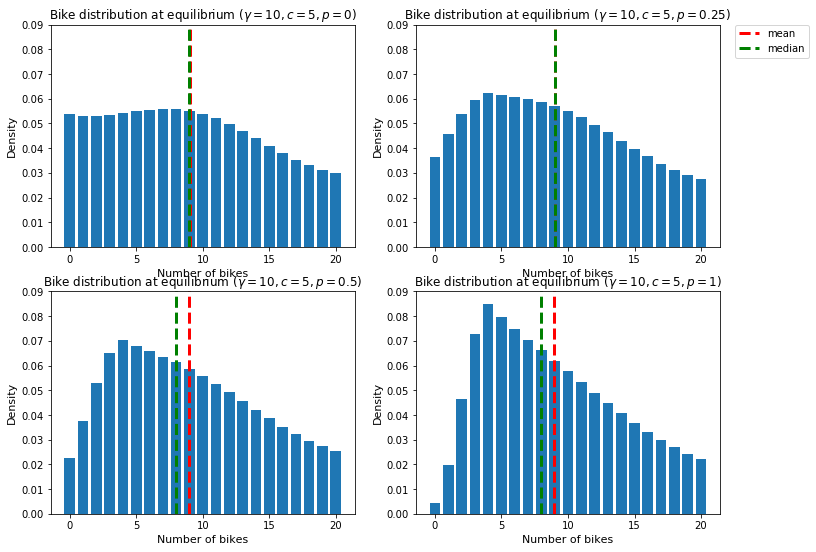}
\caption{Bike distribution for different values of $p$ when $K=20$, $\gamma=10$ and $c=5$}\label{Fig_bike_dist_c5}
\end{figure}

\begin{figure}[H]
\centering
\includegraphics[width=0.9\textwidth]{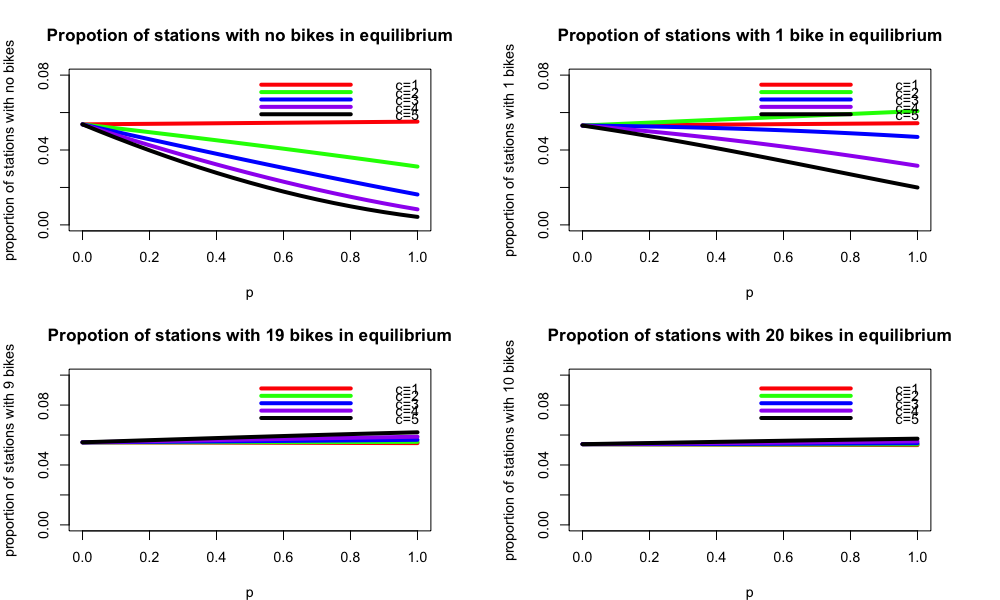}
\caption{Proportion of stations with no bikes, 1 bike, 19 bikes and 20 bikes for different value of $c$}\label{Fig_y_p_c}
\end{figure}

\begin{figure}[H]
\centering
\includegraphics[width=0.8\textwidth]{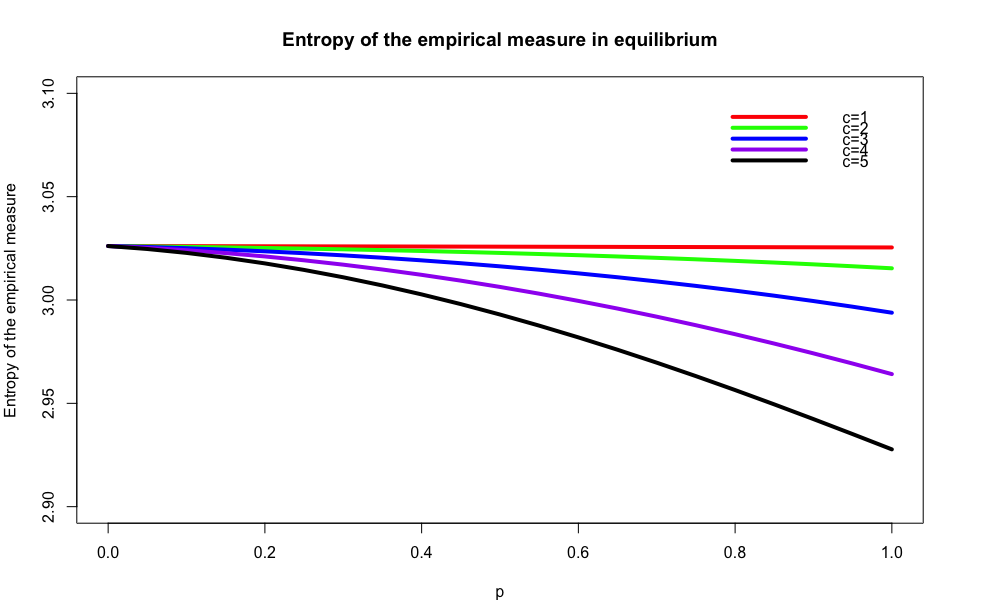}
\caption{Entropy of the empirical measure in equilibrium for different value of $c$}\label{Fig_entropy_p_c}
\end{figure}

\paragraph{Polynomial Choice Function $g(x)=x^{\alpha}$} 
Finally we then explore the polynomial choice function. This choice function represents a user sensitivity to information about bike availability that is in between the exponential choice function and the minimum choice function, as the probability of choosing stations with more bikes has a polynomial growth.

Figure \ref{Fig_bike_dist_alpha2} shows the empirical process at equilibrium for $p=0, 0.25, 0.5, 1$ and $\alpha=2$. We observe that with no information ($p=0$) the distribution looks relatively flat. As the value of $p$ increase to 1, i.e. the proportion of users with information, the proportion of empty and full stations goes down to close to 0. The effect of polynomial choice function is similar to that of exponential choice function, but less extreme.

Figure \ref{Fig_y_p_alpha} shows the surface plots of the mean field limit on the $p-\alpha$ plane with $p\in [0,1]$ and $\alpha\in [0,5]$. Particularly, we look at proportion of stations with 0,1,19 and 20 bikes. These plots show that proportion of nearly empty and full stations go down dramatically both when $p$ increase from 0 to 1 and $\alpha$ increase from 0 to 5. Its effect is similar to those from the exponential choice function.

Figure \ref{Fig_entropy_p_alpha} shows the surface plot of the entropy of the mean field limit on the $p-\alpha$ plane. We observe that the entropy goes down as $p, \theta$ increase, which is consistent with the findings in Figures \ref{Fig_y_p_alpha} and \ref{Fig_entropy_p_alpha}.

\begin{figure}[H]
\centering
\includegraphics[width=0.75\textwidth]{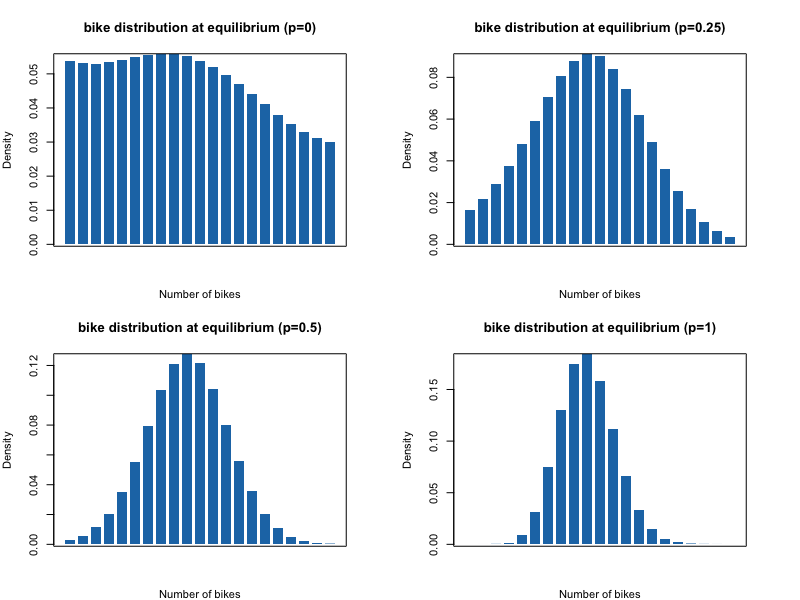}
\caption{Bike distribution for different values of $p$ when $K=20$, $\gamma=10$ and $\alpha=2$}\label{Fig_bike_dist_alpha2}
\end{figure}

\begin{figure}[H]
\centering
\includegraphics[width=0.8\textwidth]{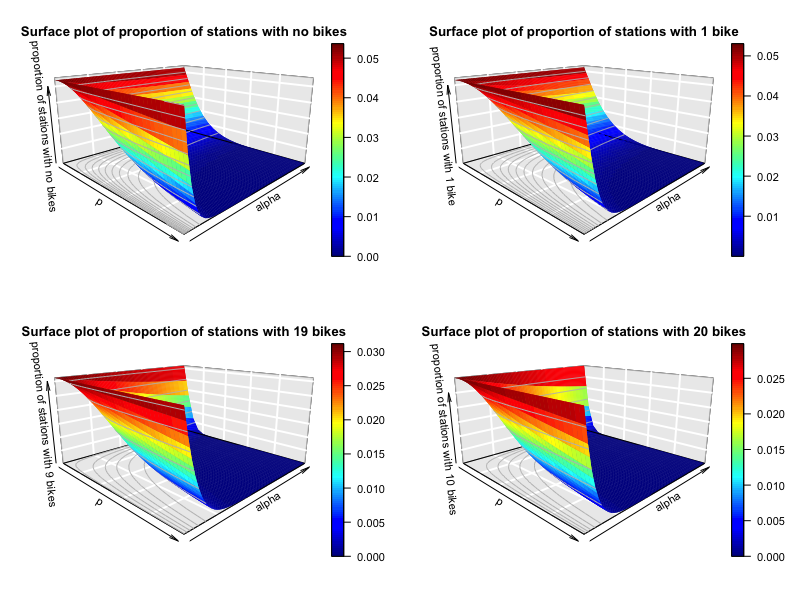}
\caption{Surface plot of $y_{\infty}(p,\alpha)$ on the $p-\alpha$ plane when $K=20$}\label{Fig_y_p_alpha}
\end{figure}

\begin{figure}[H]
\centering
\includegraphics[width=0.8\textwidth]{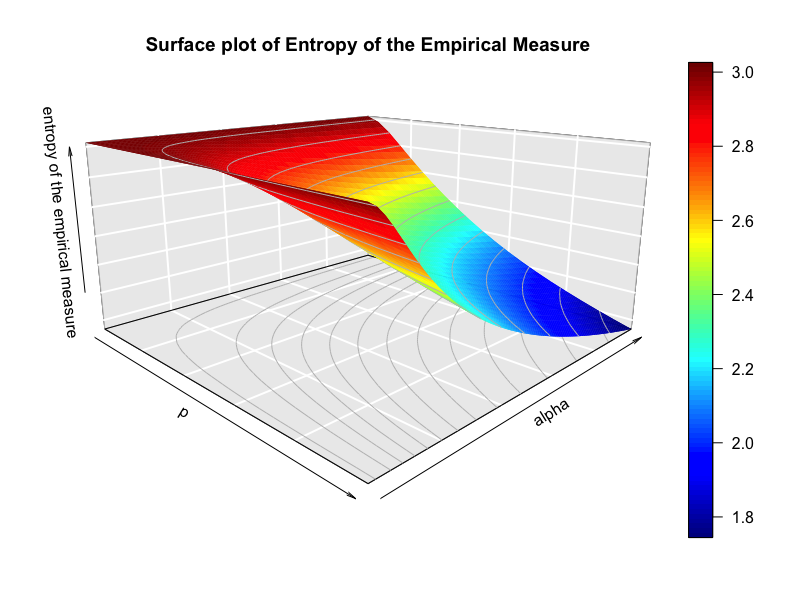}
\caption{Surface plot of the entropy of the empirical measure in equilibrium on the $p-\alpha$ plane when $K=20$}\label{Fig_entropy_p_alpha}
\end{figure}

\subsection{Non-stationary Arrivals}\label{Sec_nonstationary}

In this section, we consider the bike sharing system under non-stationary arrivals and provide numerical examples to analyze the effect of choice modeling. This is important because in real life bike sharing systems, arrival rates are often non-stationary and periodic.

\subsubsection{Fitting Arrival Rate Function}\label{fit_arrival}
Given the periodic property of the arrival rate function we observe from the historical data of CitiBike (see Figure \ref{Fig_arrival_regression}), we perform linear regression using the following Fourier basis,
\begin{equation}
\lambda (t) = \sum_{i=0}^{n} \beta_{1,j}\sin(2\pi t[j/\omega])+\beta_{2,j}\cos(2\pi t[j/\omega]),
\end{equation} 
where the unit of $t$ is hour and $\omega=24$ is the period of $\lambda(t)$.

Since the arrival rate patterns is quite different between weekdays and weekends, we will fit two separate regression on them.

\paragraph{Weekdays} 
The fitted result of $\lambda(t)$ during the weekdays for $n=5$ is the following,
\begin{eqnarray}\label{lambda_weekday}
\lambda(t)&=&-43.4\sin(2\pi t/24)-49.5\cos(2\pi t/24)-38.2\sin(4\pi t/24)-40.0\cos(4\pi t/24)\nonumber \\
& &+30.1\sin(6\pi t/24)+23.7\cos(6\pi t/24)+14.6\sin(8\pi t/24)-1.4\cos(8\pi t/24)\nonumber\\
& &-29.4\sin(10\pi t/24)+1.4\cos(10\pi t/24)+91.4
\end{eqnarray}

Table \ref{table:r^2_weekdays} shows the $R^2$ statistics for different choice of $n$. We can see that as we increase $n$, which indicates the complexity of function form, we get better approximation to the actual arrival rate function. With $n=5$ we obtain an $R^2$ statistics of 0.947, while the benefit of adding more complexity diminishes as $n$ increase from 5 to 10.

\begin{table}[ht] 
\centering 
\begin{tabular}{ *{11}{l} | c | r }
\hline
$n$ & 1 & 2 & 3 & 4 & 5 & 6 & 7 & 8 & 9 & 10\\ \hline
$R^2$ statistics & 0.41 & 0.70 & 0.84 & 0.86 & 0.947 & 0.955 & 0.962 & 0.972 & 0.973 & 0.975\\ 
\hline
\end{tabular}
\caption{$R^2$ statistics for different degree $n$ (weekdays)}\label{table:r^2_weekdays}
\end{table}

Figure \ref{Fig_arrival_regression} shows the actual vs. predicted arrival rate of customers (average number of trips during each 5 minutes) on weekdays from CitiBike data. We can see that the fitted results are very accurate in comparison to the real data.

\paragraph{Weekends}The fitted result of $\lambda(t)$ during the weekends for $n=2$ is the following,
\begin{eqnarray}\label{lambda_weekend}
\hat{\lambda}(t)&=&-43.8\sin(2\pi t/24)-39.5\cos(2\pi t/24)+6.0\sin(4\pi t/24)+6.7\cos(4\pi t/24)+58.6.\nonumber\\
\end{eqnarray}

Table \ref{table:r^2_weekends} shows the $R^2$ statistics for different degree $n$. Since the arrival rate pattern is much simpler during weekends than in weekdays, the approximation is much easier. With $n=1$ we already obtain an $R^2$ statistics of 0.955.

\begin{table}[ht]
\centering 
\begin{tabular}{ *{4}{l}| c | r }
\hline
$n$ & 1 & 2 & 3 \\ \hline
$R^2$ statistics & 0.955 & 0.977 & 0.987 \\ 
\hline
\end{tabular}
\caption{$R^2$ statistics for different degree $n$ (weekends)}\label{table:r^2_weekends} 
\end{table}

\begin{figure}[h]
\centering
\includegraphics[width=0.52\textwidth]{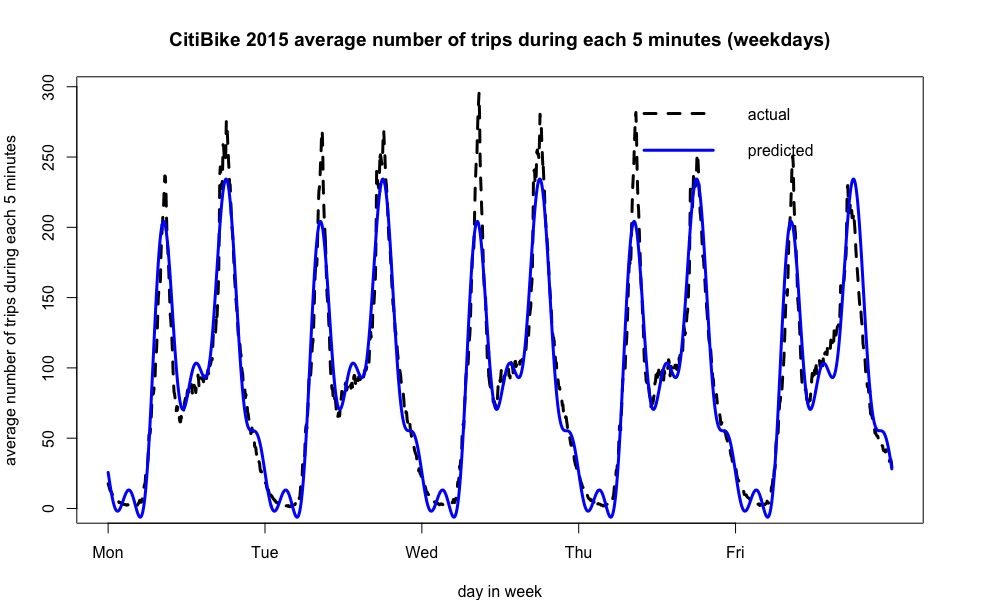}~
\includegraphics[width=0.52\textwidth]{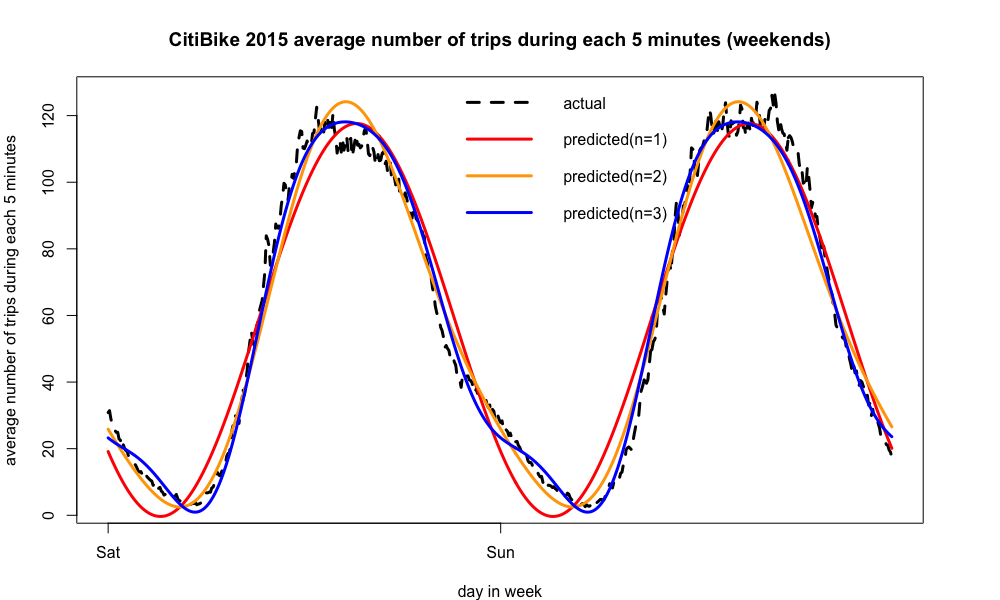}
\captionsetup{justification=centering}
\caption{Fitted Arrival Rate Function (weekdays (left) and weekends (right))}\label{Fig_arrival_regression}
\end{figure}

Figure \ref{Fig_arrival_regression} shows the actual vs. predicted arrival rate of customers (average number of trips during each 5 minutes) on weekends from CitiBike data. Again, we can see that the fitted results are very accurate in comparison to the real data for just $n=2$.

\subsubsection{Impact of giving customer information ($p$) to the empirical measure}
In the case of non-stationary arrivals, a figure is not enough to visualize the changes of bike distribution over time for us to analyze the impact of choice modeling. Thus, we create videos of the empirical measure process over time below.

We provide a video in this \href{https://cpb-us-e1.wpmucdn.com/blogs.cornell.edu/dist/3/7882/files/2020/04/bikevideo_exp.mp4}{link} the empirical measure over time under non-stationary arrival rates with exponential choice function for $p=0,0.25,0.5,1$. The arrival rates functions used in the video are from Equations (\ref{lambda_weekday}) and (\ref{lambda_weekend}) from Section \ref{fit_arrival}. At time 0, the distribution is initialized at a mass at 10 bikes per station. We observe from the video that at any given time, as $p$, the proportion of users with information, increase from 0 to 1, the proportion of empty and full stations decrease dramatically. This is consistent with our finding in Section \ref{Sec_stationary}. We also observe that, the empirical process where $p$ is larger is more centered around its mean, i.e. more stable through out the length of simulation, and the effect is getting bigger when $p$ increase from 0 to 1. This suggests that with the exponential choice function, more people having information about the system will make the system more stable (in terms of less empty/full stations).


We provide a video in this \href{https://cpb-us-e1.wpmucdn.com/blogs.cornell.edu/dist/3/7882/files/2020/04/bikevideo_min.mp4}{link} the empirical measure over time under non-stationary arrival rates with minimum choice function for $p=0,0.25,0.5,1$. The arrival rates function and the initial setting is the same as in previous video with exponential choice function. We observe that at any given time, as $p$ increase from 0 to 1, the proportion of empty stations decrease dramatically. However there is very little impact on full or nearly full stations. This is consistent with our finding in Section \ref{Sec_stationary}. We also observe that the empirical process with larger $p$ has less empty or nearly empty stations throughout the duration of the simulation, not not much difference in terms of full or nearly full stations. This suggests that the minimum choice function, more people having information about the system will give a better control on the proportion of empty or nearly empty stations, however it has very little impact on full or nearly full stations.


Now we look at a simpler example, where $K=3$ and $\lambda(t)=1+0.5\sin(t/2)$, and examine each individual component of the empirical process. That is,  we show the surface plots of the fluid limit of the proportion of stations with 0,1,2,3 bikes individually on the time-$p$ plane under non-stationary arrivals.  Here we use the exponential choice function as an example. Parameters are given as $\theta=2$ and initial distribution of bikes is $y(0)=(0.25,0.25,0.25,0.25)$.

In Figure \ref{Fig_y_p_t_NS}, we can see that as $p$ increases, the average proportion of empty stations goes down. However, the amplitude of the variation of the proportion of empty stations becomes larger, which means the system is more unstable. In terms of full stations, we see the opposite: both the average proportion of full stations and the amplitude of variation go down. This tells us that in the time-varying arrival case, we may not want to give customer too much information.

In Figure \ref{Fig_entropy_p_NS}, we show the surface plot of the entropy of the empirical process on the time-$p$ plane. We observe that the average entropy is decreasing as $p$ increase and we do not see huge changes in amplitude of the variation of entropy over time.

In summary, this example suggests that under the non-stationary arrivals setting, it might not always be better to give users more information. Whether giving people information will lower proportion of empty or full stations depends on the specific parameters setting, and policy makers of bike sharing companies should be aware of this. However, it is very easy to verify the benefit of giving users information using the mean field limit and diffusion limit results to give approximations to the mean and variance of empirical measure process under non-stationary setting.

\begin{figure}[H]
\centering
\includegraphics[width=1\textwidth]{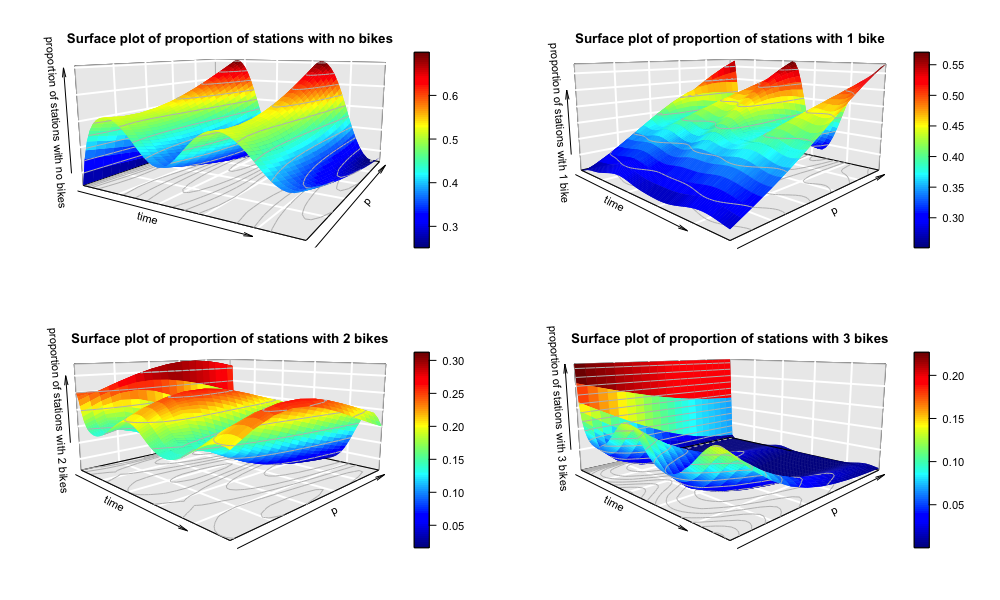}
\caption{Surface plot of $y(t,p)$ on the time-$p$ plane when $K=3$}\label{Fig_y_p_t_NS}
\end{figure}

\begin{figure}[H]
\centering
\includegraphics[width=0.8\textwidth]{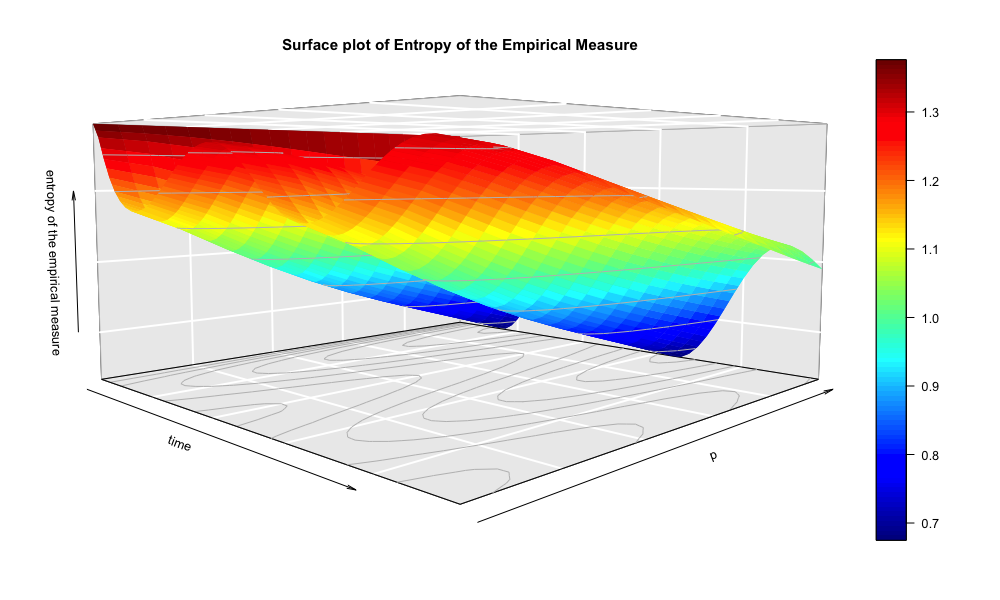}
\caption{Surface plot of entropy of the empirical measure on the time-$p$ plane when $K=3$}\label{Fig_entropy_p_NS}
\end{figure}

\subsubsection{Impact of user's sensitivity to information ($\theta$) to the empirical measure}
In this section, we explore the impact of user's sensitivity to information on the empirical measure process. Specifically, we consider the exponential choice function.

In Figure \ref{Fig_y_theta_t1_NS} we show the surface plots of the fluid limit on the time-$\theta$ plane, with $K=3$, $p=0.5$ and $\theta\in [0,5]$. The initial distribution of bikes is $y(0)=(0.25,0.25,0.25,0.25)$.
We observe that the the average proportion of empty stations slightly goes down, but the amplitude of the variation becomes larger. This suggests that as users become more sensitive to information, there exists a trade-off between system average and system variation under non-stationary setting in terms of controlling empty stations. Depending on the goal, it might not always be better when users are more sensitive to information.

In Figure \ref{Fig_entropy_theta_NS} we show the surface plot of the entropy of the empirical measure process on the time-$\theta$ plane. We observe that the average entropy and the amplitude of variation of the entropy is not monotone on $\theta$, indicating that there exists an optimal value of $\theta$ that minimize the mean (variance) of the entropy under the non-stationary setting.

In summary, this example suggests that under the non-stationary arrivals setting, it might not always be better when customers are more sensitive to information. Whether people being more sensitive to  information will lower proportion of empty or full stations depends on the specific parameters setting. There might be a trade-off between the system average performance and the system variation, in which case an optimal sensitivity parameter $\theta$ exists where the empirical measure achieves minimal entropy.

\begin{figure}[H]
\centering
\includegraphics[width=1\textwidth]{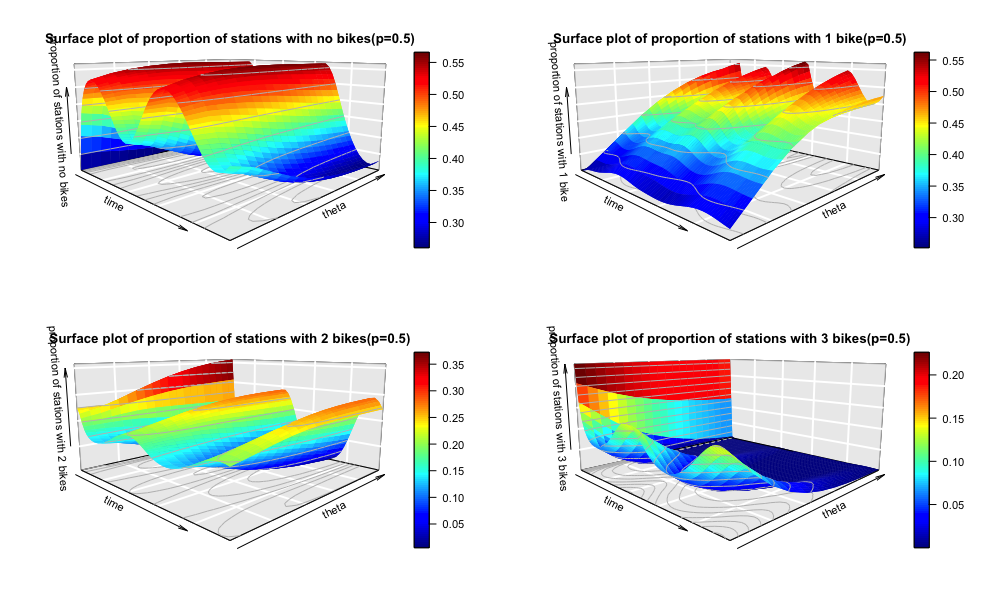}
\caption{Surface plot of $y_{(t,\theta)}$ on the time-$\theta$ plane when $K=3$ and $p=0.5$}\label{Fig_y_theta_t1_NS}
\end{figure}

\begin{figure}[H]
\centering
\includegraphics[width=0.8\textwidth]{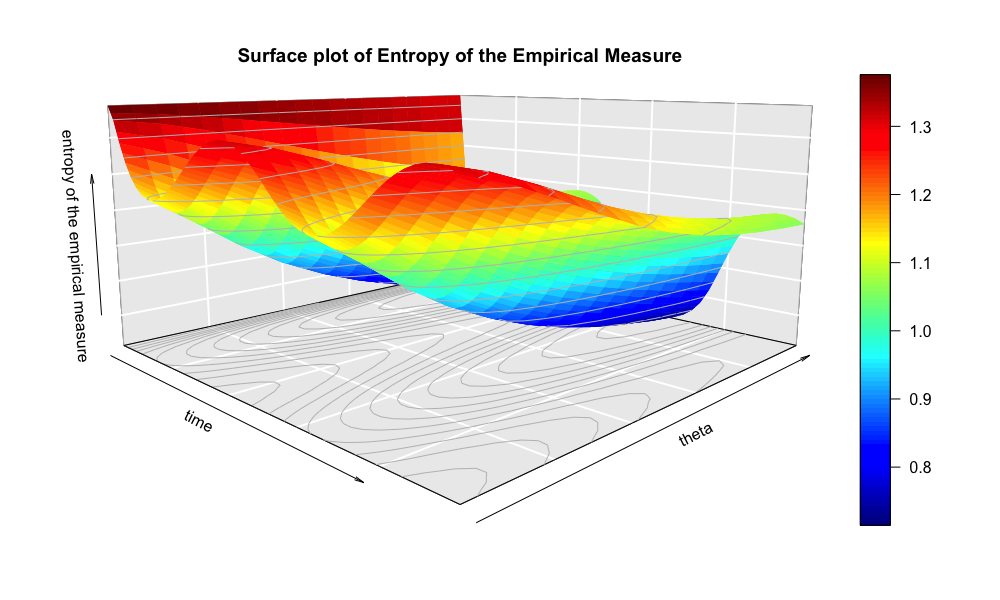}
\caption{Surface plot of entropy of the empirical measure on the time-$p$ plane when $K=3$}\label{Fig_entropy_theta_NS}
\end{figure}

\subsubsection{Impact of fleet size ($\gamma$) to the empirical measure}
In this section, we explore the impact of fleet size on the empirical measure process, as this is very importance from a system design and management point of view. 

Let $\gamma$ be the average number of bikes per station. We use the same parameter setting as in previous section, with $K=3, p=0.5$ and exponential choice function with $\theta=2$. The range of $\gamma$ is [0.5,2.5].  The initial distribution of bikes is $y(0)=(1-0.4\gamma,0,0.2\gamma,0.2\gamma)$.

In Figure \ref{Fig_y_s_t_NS} we show the surface plot of the fluid limit on the time-$\gamma$ plane, with $K=3$, $p=0.5$ and $\gamma\in [0.5,2.5]$. We observe that as fleet size (average number of bikes per station) increases, average proportion of empty stations goes down and there is not much change in the variation of proportion of empty stations over time. In terms of full stations, both the average proportion of full stations and the amplitude of its variation over time increase. This indicates that there exists an optimal fleet size $\gamma$ that minimize proportion of empty and full stations overall.

In Figure \ref{Fig_entropy_s_NS} we show the surface plot of the entropy of the empirical measure on the time-$\gamma$ plane. We observe that as fleet size increases, the average entropy increases first then slightly goes down, while the amplitude of the variation of entropy goes down. This indicates that there exists an optimal fleet size $\gamma$ that minimize the entropy of the empirical measure.

In summary, this example suggests that there exists an optimal value for fleet size that minimize the mean and variance of entropy of the empirical measure under the non-stationary setting. The optimal value can be easily numerically identified with the mean field limit results.

\begin{figure}[H]
\centering
\includegraphics[width=1\textwidth]{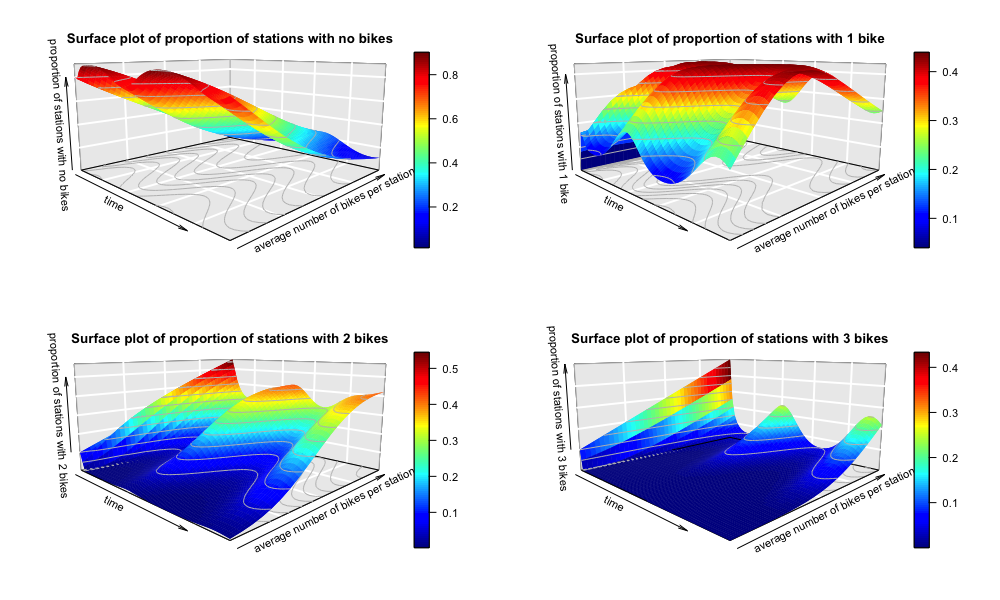}
\caption{Surface plot of $y_{(t,\gamma)}$ on the time-$\gamma$ plane when $K=3, p=0.5$ and $\theta=2$}\label{Fig_y_s_t_NS}
\end{figure}

\begin{figure}[H]
\centering
\includegraphics[width=0.8\textwidth]{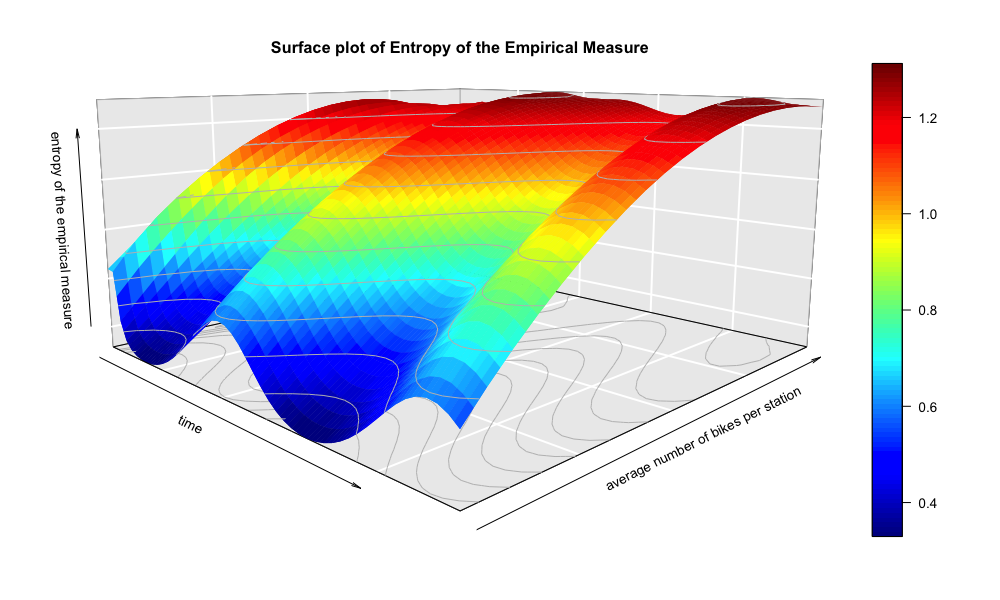}
\caption{Surface plot of entropy of the empirical measure on the time-$p$ plane when $K=3$}\label{Fig_entropy_s_NS}
\end{figure}

\section{Conclusion}\label{conc}

In this paper, we construct a stochastic queueing model with customer choice modeling for the empirical measure process of a finite-capacity bike sharing system. We consider a fraction $p$ of the customers having information to the bike availability of the system, which impacts their arrival rate to the stations through a choice function. 

First, we prove a mean field limit and a central limit theorem for an empirical process of the number of stations with $k$ bikes, when station capacities are the same across the  network. We also defined a new process called the ratio process in case when station capacities are different, and it characterizes the proportion of stations whose bike availability ratio lies within a particular partition of the interval [0,1]. We prove similar mean field limit and central limit theorem results for the ratio process. These asymptotic results provide great insights into the mean, variance and sample path dynamics of large scale bike sharing systems.

Second, we provide steady state analysis of the mean field limit of both the empirical measure process and the ratio process. We explicitly compute the equilibrium of the mean field limit and prove that interchanging of limits ($\lim_{t\rightarrow \infty}\lim_{N\rightarrow \infty}=\lim_{N\rightarrow \infty}\lim_{t\rightarrow \infty}$) holds for the mean field limit in equilibrium.

Finally, we give numerical examples to provide insights on how the fraction of customers with information ($p$), different choice models, customers' sensitivity to information and fleet size impact the system performance, under both stationary arrivals setting and non-stationary arrivals setting. Our examples show that generally, under stationary arrivals, the more customers with information, the better the system performance is. Under non-stationary settings, there might exists a trade-off between average system performance and its variation over time, and whether more information is more beneficial depends on the specific parameter setting. The form of choice models also plays role in how it impacts the system in terms of controlling empty and full stations.

Despite our analysis, there are many future directions for research.
\begin{enumerate}
\item The first direction is showing that the interchanging of limits holds for the diffusion limit as well. This requires proving exponential stability of the mean field limit.
\item We are also interested in constructing a stochastic queueing model for bike sharing systems that takes into account geographic location of stations. In this case we would consider customers arriving via a spatial Poisson process and would consider not only available bikes but also distance to the stations for picking up bikes.
\item Our model can also be adopted to study the effect of customer incentives or reward programs. These programs give riders the opportunity to move bikes from stations to stations, with the hope of improving the overall availability of bikes and docks across the network. For example, some major bike sharing systems such as CitiBike has so-called "Bike Angels", which is a reward program to riders who help move bikes from full stations to empty stations. This may potentially be a more cost-effective way to rebalancing than hiring employees to move bikes. See \citet{10.1145/3209811.3209866} for more discussion on the topic. In this case, the effect of "Bike Angels" can be modeled similarly to the choice model in this paper as follows,
\begin{itemize}
\item Rate of picking up bikes at station $i$:
\begin{eqnarray*}
Q(X_i(t),X_i(t)-1)&=&\left(\underbrace{(1-p)\lambda_{i}}_{\text{no information arrival}}+\underbrace {p(1-\alpha)\lambda_{i}N\frac{g(X_i(t))}{\sum_{j=1}^{N}g(X_{j}(t))}}_{\text{with information arrival}}\right.\nonumber\\
& &\left. +\underbrace {p\alpha\lambda_{i}N\frac{h(X_i(t))}{\sum_{j=1}^{N}h(X_{j}(t))}}_{\text{bike angles}}\right)\cdot \underbrace{\mathbf{1}\{X_{i}(t)>0\}}_{\text{station $i$ not empty}}
\end{eqnarray*}
\item Rate of dropping off bikes at station $i$ : 
\begin{eqnarray*}
Q(X_i(t),X_i(t)+1)&=& \underbrace{\mu}_{\text{service rate}} \cdot \left(\underbrace{P_{i}(1-p\alpha)}_{\text{routing probability}}+\underbrace{P_{i}p\alpha N\frac{h(-X_i(t))}{\sum_{j=1}^{N}h(-X_{j}(t))}}_{\text{routing probability for bike angles}}\right)\cdot\nonumber\\
& & \underbrace{\left(M-\sum_{k=1}^{N}X_{k}(t)\right)}_{\text{\# of bikes in circulation}}\cdot \underbrace{\mathbf{1}\{X_{i}(t)<K_{i}\}}_{\text{station $i$ not full}}
 \end{eqnarray*}
\end{itemize}
where $g,h$ are arbitrary functions that are non-decreasing
and non-negative.

\end{enumerate}
We intend to pursue these extensions in future work.


\section{Appendix}\label{App}
\begin{proof}[Proof of Theorem \ref{fluid_limit}]
The proof we give is obtained by Doob's inequality for martingales and Gronwall's lemma. We use Proposition \ref{bound}, Proposition \ref{Lipschitz}, and Proposition \ref{drift} in the proof, and they are stated after Theorem \ref{fluid_limit}.

Since $Y_{t}^{N}$ is a semi-martingale, we have the following decomposition of $Y_{t}^{N}$,
\begin{equation}\label{semiY}
Y_{t}^{N}=\underbrace{Y_0^{N}}_{\text{initial condition}}+\underbrace{M_t^{N}}_{\text{martingale}}+\int_{0}^{t}\underbrace{\beta(Y_s^{N})}_{\text{drift term}}ds
\end{equation}
where
$Y_{0}^{N}$ is the initial condition and $M_{t}^{N}$ is a independent family of martingales.  Moreover, $\int_{0}^{t}\beta(Y_s^{N})ds$ is the integral of the drift term where the drift term is given by $\beta: [0,1]^{K+1}\rightarrow \mathbb{R}^{K+1}$ or
\begin{eqnarray}
\beta(y)&=&\sum_{x\neq y}(x-y)Q(y,x)\nonumber\\
&=& \sum_{n=0}^{K}\left[ \left((1-p)+p\frac{g(n)}{\sum_{j=0}^{K}y_{j}g(j)}\right)(\mathbf{1}_{n-1}-\mathbf{1}_{n})\lambda\mathbf{1}_{n>0}\right.\nonumber\\
& &\left.+\left(\frac{M}{N}-\sum_{j=0}^{K} jy_{j} \right)(\mathbf{1}_{n+1}-\mathbf{1}_{n})\mathbf{1}_{n<K}\right] y_n.
\end{eqnarray}

We want to compare the empirical measure $Y_{t}^{N}$ with the mean field limit $y_{t}$ defined by
\begin{equation}
y_t=y_0+\int_{0}^{t}b(y_s)ds.
\end{equation}

Let $|\cdot|$ denote the Euclidean norm in $\mathbb{R}^{K+1}$, then
\begin{eqnarray}
\left|Y_{t}^{N}-y_{t}\right|&=&\left|Y_{0}^{N}+M_{t}^{N}+\int_{0}^{t}\beta
(Y_{s}^{N})ds-y_0-\int_{0}^{t}b(y_{s})ds\right|\nonumber \\
& =& \left|Y_{0}^{N}-y_0+M_{s}^{N}+\int_{0}^{t}\left(\beta
(Y_{s}^{N})-b(Y_{s}^{N})\right)ds+\int_{0}^{t}(b(Y_{s}^{N})-b(y_{s}))ds\right|.\nonumber \\
\end{eqnarray}
Now define the random function $f^{N}(t)=\sup_{s\leq t}\left|Y_s^{N}-y_s\right|$, we have
$$f^{N}(t)\leq |Y_0^{N}-y_0|+\sup_{s\leq t}|M_s^{N}|+\int_0^t|\beta(Y_s^{N})-b(Y_s^{N})|ds+\int_{0}^{t}|b(Y_s^{N})-b(y_s)|ds.$$
By Proposition \ref{Lipschitz}, $b(y)$ is Lipschitz with respect to Euclidean norm. Let $L$ be the Lipschitz constant of $b(y)$, then
\begin{eqnarray}
f^{N}(t)&\leq& |Y_0^{N}-y_0|+\sup_{s\leq t}|M_s^{N}|+\int_0^t|\beta(Y_s^{N})-b(Y_s^{N})|ds+\int_{0}^{t}|b(Y_s^{N})-b(y_s)|ds \nonumber \\
&\leq& |Y_0^{N}-y_0|+\sup_{s\leq t}|M_s^{N}|+\int_0^t|\beta(Y_s^{N})-b(Y_s^{N})|ds+L\int_{0}^{t}|Y_s^{N}-y_s|ds \nonumber \\
&\leq& |Y_0^{N}-y_0|+\sup_{s\leq t}|M_s^{N}|+\int_0^t|\beta(Y_s^{N})-b(Y_s^{N})|ds+L\int_{0}^{t}f^{N}(s)ds.
\end{eqnarray}

By Gronwall's lemma (See \citet{ames1997inequalities}), 
\begin{equation}
f^{N}(t) \leq \left(|Y_0^{N}-y_0|+\sup_{s\leq t}|M_s^{N}|+\int_0^t|\beta(Y_s^{N})-b(Y_s^{N})|ds\right)e^{Lt}.
\end{equation}

Now  to bound $f^{N}(t)$ term by term, we define function $\alpha: [0,1]^{K+1}\rightarrow \mathbb{R}^{K+1}$ as
\begin{eqnarray}
\alpha(y)&=&\sum_{x\neq y}|x-y|^2Q(y,x)\nonumber\\
&=&\frac{1}{N}\sum_{n}\sum_{r}\left[\frac{1}{rNR_{\max}}(\mathbf{1}_{(r,n-1)}+\mathbf{1}_{(r,n)})\mathbf{1}_{n>0}  \right.\nonumber \\& &+
\left. \left(\frac{M}{N}-\sum_{n'}\sum_{r'}n'y(r',n')\right)(\mathbf{1}_{(r,n+1)}+\mathbf{1}_{(r,n)})\mathbf{1}_{n<K} \right] \cdot y(r,n) 
\end{eqnarray}

and consider  the following four sets 
\begin{eqnarray}
\Omega_0 &=& \{|Y_0^{N}-y_0|\leq \delta \}, \\
\Omega_1 &=& \left\{\int_0^{t_0}|\beta(Y_s^{N})-b(Y_s^{N})|ds\leq \delta \right\}, \\
\Omega_2 &=& \left\{\int_0^{t_0}\alpha(Y_t^{N})dt \leq A(N)t_0 \right\}, \\
\Omega_3 &=& \left\{\sup_{t\leq t_0}|M_t^{N}|\leq \delta \right\} ,
\end{eqnarray}
where $\delta=\epsilon e^{-Lt_0}/3$. Here the set $\Omega_{1}$ is to bound the initial condition, the set $\Omega_{2}$ is to bound the drift term $\beta$ and the limit of drift term $b$, and  the sets $\Omega_{2},\Omega_{3}$ are to bound the martingale $M_{t}^{N}$. Therefore on the event $\Omega_0\cap \Omega_1\cap \Omega_3$,
\begin{equation}\label{e}
f^{N}(t_0)\leq 3\delta e^{Lt_0}=\epsilon.
\end{equation}

Since $\lim_{N\rightarrow \infty}\frac{M}{N}=\gamma$ and $\lim_{N\rightarrow \infty}NR_{\max}=\frac{1}{\Lambda}$, we can
choose large enough $N$ such that
$$\frac{M}{N}\leq 2\gamma,\quad NR_{\max}\geq \frac{1}{2\Lambda}.$$
Thus
\begin{eqnarray}
\alpha(y)&=&\sum_{x\neq y}|x-y|^2Q(y,x)\nonumber\\
&=& \frac{1}{N}\sum_{n=0}^{K}\left[ \left((1-p)+p\frac{g(n)}{\sum_{j=0}^{K}y_{j}g(j)}\right)(\mathbf{1}_{n-1}+\mathbf{1}_{n})\lambda\mathbf{1}_{n>0}\right.\nonumber\\
& &\left.+\left(\frac{M}{N}-\sum_{j=0}^{K} jy_{j} \right)(\mathbf{1}_{n+1}+\mathbf{1}_{n})\mathbf{1}_{n<K}\right] y_n\nonumber \\
&\leq & \frac{1}{N}\sum_{n=0}^{K}\left[ \left((1-p)+p\frac{C_{\max}}{C_{\min}}\right)\cdot 2\lambda+2\gamma\cdot 2\right] y_n\nonumber \\
&=&\frac{2}{N}\left[ \left((1-p)+p\frac{C_{\max}}{C_{\min}}\right)\lambda+2\gamma\right] \nonumber \\
&\sim & O(\frac{1}{N}).
\end{eqnarray}
Consider the stopping time $$T=t_0\wedge \inf \left\{t\geq 0:\int _{0}^{t}\alpha(Y_s^{N})ds>A(N)t_0\right\},$$
by Proposition \ref{bound},
$$\mathbb{E}\left(\sup_{t\leq T}|M_{t}^{N}|^2\right)\leq 4\mathbb{E}\int_{0}^{T}\alpha(Y_{t}^{N})dt\leq 4A(N)t_0. $$
On $\Omega_2$, we have $T =t_0$, so 
$\Omega_2 \cap \Omega_3^{c}\subset \{\sup_{t\leq T}|M^{N}_t|>\delta\}$. By Chebyshev's inequality, 
\begin{equation}
\mathbb{P}(\Omega_2 \cap \Omega_3^{c})\leq \mathbb{P}\left(\sup_{t\leq T}|M^{N}_t|>\delta\right)\leq \frac{\mathbb{E}\left(\sup_{t\leq T}|M_{t}^{N}|^2\right)}{\delta^2}\leq 4A(N)t_0/\delta^2.
\end{equation}
Thus by Equation (\ref{e}), we have the following result,
\begin{equation}
\begin{split}
\mathbb{P}\left(\sup_{t\leq t_0}|Y_t^{N}-y_t|>\epsilon\right)&\leq \mathbb{P}(\Omega_0^c\cup \Omega_1^c\cup \Omega_3^c)\\
&\leq \mathbb{P}(\Omega_2 \cap \Omega_3^{c})+\mathbb{P}(\Omega_0^{c}\cup \Omega_1^{c}\cup \Omega_2^{c})\\
&\leq 4A(N)t_0/\delta^2+\mathbb{P}(\Omega_0^{c}\cup \Omega_1^{c}\cup \Omega_2^{c})\\
&=36A(N)t_0 e^{2Lt_0}/\epsilon^2+\mathbb{P}(\Omega_0^{c}\cup \Omega_1^{c}\cup \Omega_2^{c}).
\end{split}
\end{equation}
Let $A(N)=\frac{4(C+\gamma)}{N}$, then $\Omega_{2}^{c}=\emptyset$.
And since $Y_0^N\xrightarrow{p} y_0$,   $\lim_{N\rightarrow \infty}\mathbb{P}(\Omega_{2}^{c})=0$. Therefore we have $$\lim_{N\rightarrow \infty}\mathbb{P}\left(\sup_{t\leq t_0}|Y_t^{N}-y_t|>\epsilon\right)=\lim_{N\rightarrow\infty}\mathbb{P}(\Omega_{1}^{c}) .$$
By Proposition \ref{drift}, $\lim_{N\rightarrow\infty}\mathbb{P}(\Omega_{1}^{c})=0$.
Thus, we proved the final result
$$\lim_{N\rightarrow \infty}\mathbb{P}\left(\sup_{t\leq t_0}|Y_t^{N}-y_t|>\epsilon\right)=0.$$\
\end{proof}

\begin{proposition}[Bounding martingales]\label{bound}
For any stopping time $T$ such that $\mathbb{E}(T)<\infty$, we have
\begin{equation}
\mathbb{E}\left(\sup_{t\leq T}|M_{t}^{N}|^2\right)\leq 4\mathbb{E}\int_{0}^{T}\alpha(Y_{t}^{N})dt.
\end{equation}  
\begin{proof}
Let $\tilde{\mu}$ be the jump measure of $Y_{t}^{N}$, and $\nu$ be its compensator, defined on $(0,\infty)\times [0,1]$  by 
\begin{eqnarray}
\tilde{\mu}=\sum_{t:Y_{t}^{N}\neq Y_{t-}^{N}}\delta(t,Y_{t}^{N}), & & \nu(dt,B)=Q(Y_{t-}^{N},B)dt \quad \forall B\in \mathcal{B}([0,1]).
\end{eqnarray}
Let $\tilde{Y}_{m,t}^{N}$ be the jump chain of $Y_{t}^{N}$, $J_{m}$ be the jump time, then we have for any $t\in [0,\infty)$, $J_{n}\leq t<J_{n+1}$ for some $n\geq 0$.
The martingale $M_{t}^{N}$ can be written as
\begin{eqnarray}
M_{t}^{N}&=&Y^{N}_{t}-Y_{0}^{N}-\int_{0}^{t}\beta(Y_{s}^{N})ds \nonumber  \\
&=&\sum_{m=0}^{n-1}(\tilde{Y}^{N}_{m+1}-\tilde{Y}^{N}_{m})-\int_{0}^{t}\int_{0}^{1}(y-Y_{s-}^{N})Q(Y_{s-}^{N},dy)ds\nonumber  \\
&=&\int_{0}^{t}\int_{0}^{1}(y-Y_{s-}^{N})\tilde{\mu}(ds,dy)-\int_{0}^{t}\int_{0}^{1}(y-Y_{s-}^{N})\nu(ds,dy)\nonumber  \\
&=&\int_{0}^{t}\int_{0}^{1}(y-Y_{s-}^{N})(\tilde{\mu}-\nu)(ds,dy).
\end{eqnarray}
Note the following identity
\begin{eqnarray}
\left(M_{t}^{N}\right)^2&=&2\int_{0}^{t}\int_{0}^{1}M_{s-}(y-Y_{s-}^{N})(\tilde{\mu}-\nu)(ds,dy)+\int_{0}^{t}\int_{0}^{1}(y-Y_{s-}^{N})^{2}\tilde{\mu}(ds,dy).
\end{eqnarray}
This can be established by verifying that the jumps of the left and right hand sides agree, and that their derivatives agree between jump times. Then we can write
\begin{equation}\label{martingale}
\left(M_{t}^{N}\right)^2=N_{t}^{N}+\int_{0}^{t}\alpha(Y_{t}^{N})ds
\end{equation}
where
\begin{equation}
N_{t}^{N}=\int_{0}^{t}\int_{0}^{1}H(s,y)(\tilde{\mu}-\nu)(ds,dy),
\end{equation}
and
\begin{equation}
H(s,y)=2M_{s-}(y-Y_{s-}^{N})+(y-Y_{s-}^{N})^{2}.
\end{equation}
Consider the previsible process
\begin{equation}
H_{2}(t,y)=H(t,y)\mathbf{1}_{\{t\leq T\wedge T_{n}\}}
\end{equation}
where $T_{n}=\inf\{t\geq 0: \beta(Y_{t}^{N})>n\}\wedge n$. Then
\begin{equation}
N_{T\wedge T_{n}}=\int_{0}^{\infty}\int_{0}^{1}H_{2}(t,y)(\tilde{\mu}-\nu)(dt,dy),
\end{equation}
and
\begin{eqnarray}
\mathbb{E}\int_{0}^{\infty}\int_{0}^{1}|H_{2}(s,y)|\nu(ds,dy)&=&\mathbb{E}\int_{0}^{T\wedge T_{n}}\int_{0}^{1}|2M_{s-}(y-Y_{s-}^{N})+(y-Y_{s-}^{N})^{2}|\nu(ds,dy) \nonumber \\
&= & \mathbb{E}\int_{0}^{T\wedge T_{n}}\left(2|M_{t}^{N}|\beta(Y_{t}^{N})+\alpha(Y_{t}^{N})\right)dt \nonumber \\
& \leq & \mathbb{E}\int_{0}^{T\wedge T_{n}}2(2+n^2)ndt+\mathbb{E}\int_{0}^{T}\alpha(Y_{t}^{N})dt \nonumber \\
&\leq &2n^4+4n^2+\frac{4(C+\gamma)}{N}\mathbb{E}(T)<\infty.
\end{eqnarray}
By Theorem 8.4 in \citet{darling2008differential}, we have that $N^{T\wedge T_{n}}$ is a martingale. Replace $t$ by $T\wedge T_{n}$ in Equation~(\ref{martingale}) and take expectation to obtain
\begin{equation}
\mathbb{E}(|M_{T\wedge T_{n}}|^2)=\mathbb{E}(N_{T\wedge T_{n}})+\mathbb{E}\int_{0}^{T\wedge T_{n}}\alpha(Y_{t}^{N})dt.
\end{equation}
Since $N^{T\wedge T_{n}}$ is a martingale, 
\begin{equation}
\mathbb{E}(N_{T\wedge T_{n}})=\mathbb{E}(N_0)=0.
\end{equation}
Apply Doob's $L^2$-inequality to the martingale $M^{T\wedge T_{n}}$ to obtain
\begin{eqnarray}
\mathbb{E}\left(\sup_{t\leq T\wedge T_{n} }|M_{t}|^2\right)&\leq& 4\mathbb{E}(|M_{T\wedge T_{n}}|^2)\nonumber \\
&=& 4\mathbb{E}\int_{0}^{T\wedge T_{n}}\alpha(Y_{t}^{N})dt.
\end{eqnarray}
\end{proof}
\end{proposition}

\begin{proposition}[Asymptotic Drift is Lipschitz]\label{Lipschitz}
The drift function $b(y)$ given in Equation (\ref{eqn:b}) is a Lipschitz function with respect to the Euclidean norm in $\mathbb{R}^{K+1}$. 
\begin{proof}
Denote $|\cdot|$ the Euclidean norm in $\mathbb{R}^{K+1}$.  Consider $y,\tilde{y}\in [0,1]^{K+1}$,
\begin{eqnarray}
|b(y)-b(\tilde{y})|&\leq& 2\lambda(1-p)|y-\tilde{y}|+\lambda p\left(\sum_{k=0}^{K}\left|\frac{y_k g(k)}{\sum_{j=0}^{K}y_{j}g(j)}-\frac{\tilde{y}_k g(k)}{\sum_{j=0}^{K}\tilde{y}_{j}g(j)}\right|^2\right)^{1/2} +2\gamma |y-\tilde{y}|\nonumber \\
&\leq & 2(\lambda(1-p)+\gamma)|y-\tilde{y}|+\lambda p \left(\sum_{k=0}^{K}\left(\frac{g(k)(y_{k}-\tilde{y}_{k})}{\min_{0\leq j\leq K}g(j)}\right)^2\right)^{1/2}.
\end{eqnarray}
Since $\{g_j\}$ are non-negative, there exists constants $C_{\min}, C_{\max}>0$ such that
\begin{equation}
C_{\min}\leq \min_{0\leq j\leq K}g(j)\leq \max_{0\leq j\leq K}g(j)\leq C_{\max}.
\end{equation}
Then
\begin{eqnarray}
|b(y)-b(\tilde{y})|&\leq& 2(\lambda(1-p)+\gamma)|y-\tilde{y}|+\lambda p \frac{C_{\max}}{C_{\min}}\left(\sum_{k=0}^{K}\left(y_{k}-\tilde{y}_{k}\right)^2\right)^{1/2}\nonumber\\
&=& 2\left(\lambda(1-p)+\lambda p  \frac{C_{\max}}{C_{\min}}+\gamma\right)|y-\tilde{y}|
\end{eqnarray}
which proves that $b(y)$ is Lipschitz with respect to Euclidean norm in $\mathbb{R}^{K+1}$.
\end{proof}
\end{proposition}

\begin{proposition}[Drift is Asymptotically Close to Lipschitz Drift]\label{drift}
Under the assumptions of Theorem \ref{fluid_limit}, we have for any $\epsilon>0$ and $s\geq 0$,
$$\lim_{N\rightarrow \infty}P(|\beta(Y_s^{N})-b(Y_s^{N})|>\epsilon)= 0.$$
\begin{proof}
\begin{eqnarray}
\left|\beta(Y^{N}_{s})-b(Y^{N}_{s})\right|&=&\left|\sum_{n=0}^{K}\left(\frac{M}{N}-\gamma\right)(\mathbf{1}_{n+1}-\mathbf{1}_{n})\mathbf{1}_{n<K}Y^{N}_{s}(n)\right|\nonumber\\
&\leq & 2\left|\frac{M}{N}-\gamma\right|\left|\sum_{n=0}^{K}Y^{N}_{s}(n) \right|\nonumber\\
&=& 2\left|\frac{M}{N}-\gamma\right|\rightarrow 0.
\end{eqnarray}
\end{proof}
\end{proposition}

\begin{lemma}\label{martigale_brackets_choice}
$\sqrt{N}M_{t}^{N}$ is a family of martingales independent of $D_{0}^{N}$ with Doob-Meyer brackets given by
\begin{eqnarray}
\boldlangle \sqrt{N}M^{N}(k)\boldrangle_{t}&=&\int_{0}^{t}(\beta_{+}(Y_{s}^{N})(k)+\beta_{-}(Y_{s}^{N})(k))ds,\nonumber\\
\boldlangle \sqrt{N}M^{N}(k),\sqrt{N}M^{N}(k+1)\boldrangle_{t}&=&-\int_{0}^{t}\left[\left((1-p)+p\frac{g(k+1)}{\sum_{j=0}^{K}jy_{j}}\right)\lambda Y_{s}^{N}(k+1)\right.\nonumber\\
& & \left.+\left(\frac{M}{N}-\sum_{j=0}^{K}jY_{s}^{N}(j)\right)Y_{s}^{N}(k)\right] ds \quad \text{for } k<K,\nonumber\\
\boldlangle \sqrt{N}M^{N}(k),\sqrt{N}M^{N}(j)\boldrangle_{t}&=&0\quad \text{for } |k-j|>1.
\end{eqnarray}
\end{lemma}
\begin{proof}
By Dynkin's formula,
\begin{equation}
\begin{split}
\boldlangle \sqrt{N}M^{N}(k)\boldrangle_{t}=&\int_{0}^{t}N\sum_{x\neq Y_{s}^{N}}|x(k)-Y_{s}^{N}(k)|^2 Q(Y_{s}^{N},x)ds\\
=&N\int_{0}^{t}\alpha(Y_{s}^{N})(k)ds\\
=&\int_{0}^{t}\sum_{r} \left[\frac{1}{rNR_{\max}}\left(Y_{s}^{N}(r,k+1)\mathbf{1}_{k<K}+Y_{s}^{N}(r,k)\mathbf{1}_{k>0}\right)\right.\\
&+\left.\left(\frac{M}{N}-\sum_{n'}\sum_{r'}n'Y_{s}^{N}(r',n')\right)\left(Y_{s}^{N}(r,k)\mathbf{1}_{k<K}+Y_{s}^{N}(r,k-1)\mathbf{1}_{k>0}\right)\right]ds\\
=&\int_{0}^{t}(\beta_{+}(Y_{s}^{N})(k)+\beta_{-}(Y_{s}^{N})(k))ds.
\end{split}
\end{equation}

To compute  $\boldlangle \sqrt{N}M^{N}(k),\sqrt{N}M^{N}(k+1)\boldrangle_{t}$ for $k<K$, since
\begin{equation}
\begin{split}
&\boldlangle M^{N}(k)+M^{N}(k+1)\boldrangle_{t}\\
=&\int_{0}^{t}\sum_{x\neq Y_{s}^{N}}\left|x(k)+x(k+1)-Y_{s}^{N}(k)-Y_{s}^{N}(k+1)\right|^2 Q(Y_{s}^{N},x)ds\\
=&\frac{1}{N}\int_{0}^{t}\sum_{r} \left[\frac{1}{rNR_{\max}}\left(Y_{s}^{N}(r,k+2)\mathbf{1}_{k<K-1}+Y_{s}^{N}(r,k)\mathbf{1}_{k>0}\right)\right.\\
&+\left.\left(\frac{M}{N}-\sum_{n'}\sum_{r'}n'Y_{s}^{N}(r',n')\right)\left(Y_{s}^{N}(r,k+1)\mathbf{1}_{k<K-1}+Y_{s}^{N}(r,k-1)\mathbf{1}_{k>0}\right)\right]ds.
\end{split}
\end{equation}
We have that
\begin{equation}
\begin{split}
&\boldlangle \sqrt{N}M^{N}(k),\sqrt{N}M^{N}(k+1)\boldrangle_{t}\\
=&\frac{N}{2}\left[\boldlangle M^{N}(k)+M^{N}(k+1)\boldrangle_{t}-\boldlangle M^{N}(k)\boldrangle_{t}-\boldlangle M^{N}(k+1)\boldrangle_{t}\right]\\
=&\frac{1}{2}\int_{0}^{t}\sum_{r} \left[\frac{1}{rNR_{\max}}\left(Y_{s}^{N}(r,k+2)\mathbf{1}_{k<K-1}+Y_{s}^{N}(r,k)\mathbf{1}_{k>0}\right)\right.\\
&+\left.\left(\frac{M}{N}-\sum_{n'}\sum_{r'}n'Y_{s}^{N}(r',n')\right)\left(Y_{s}^{N}(r,k+1)\mathbf{1}_{k<K-1}+Y_{s}^{N}(r,k-1)\mathbf{1}_{k>0}\right)\right]ds\\
&-\frac{1}{2}\int_{0}^{t}(\beta_{+}(Y_{s}^{N})(k)+\beta_{+}(Y_{s}^{N})(k+1)+\beta_{-}(Y_{s}^{N})(k)+\beta_{-}(Y_{s}^{N})(k+1))ds\\
=&-\int_{0}^{t}\sum_{r}\left[\frac{1}{rNR_{\max}}Y_{s}^{N}(r,k+1)+\left(\frac{M}{N}-\sum_{n'}\sum_{r'}n'Y_{s}^{N}(r',n')\right)Y_{s}^{N}(r,k)\right] ds.
\end{split}
\end{equation}
When $|k-j|>1$, $M^{N}(k)$ and $M^{N}(j)$ are independent, thus 
\begin{equation}
\boldlangle \sqrt{N}M^{N}(k),\sqrt{N}M^{N}(j)\boldrangle_{t}=0.
\end{equation}
\end{proof}

\begin{proposition}\label{driftbound_choice}
For any $s\geq 0$,
\begin{equation}
\limsup_{N\rightarrow \infty}\sqrt{N}\left|\beta(Y_{s}^{N})-b(Y_{s}^{N})\right|<\infty.
\end{equation}
\begin{proof}
\begin{eqnarray}
\limsup_{N\rightarrow \infty}\sqrt{N}\left|\beta(Y^{N}_{s})-b(Y^{N}_{s})\right|&=&\limsup_{N\rightarrow \infty}\sqrt{N}\left|\sum_{n=0}^{K}\left(\frac{M}{N}-\gamma\right)(\mathbf{1}_{n+1}-\mathbf{1}_{n})\mathbf{1}_{n<K}Y^{N}_{s}(n)\right|\nonumber\\
&\leq & \limsup_{N\rightarrow \infty}2\sqrt{N}\left|\frac{M}{N}-\gamma\right|\left|\sum_{n=0}^{K}Y^{N}_{s}(n) \right|\nonumber\\
&=&\limsup_{N\rightarrow \infty} 2\sqrt{N}\left|\frac{M}{N}-\gamma\right|\nonumber\\
&<& \infty.
\end{eqnarray} 
\end{proof}
\end{proposition}

\begin{lemma}(Finite Horizon Bound) \label{L2bound_choice}
For any $T\geq 0$, if $$\limsup_{N\rightarrow \infty}\mathbb{E}\left(|D_{0}^{N}|^2 \right) < \infty ,$$ then we have $$\limsup_{N\rightarrow \infty}\mathbb{E}\left(\sup_{0\leq t\leq T}|D_{t}^{N}|^2 \right) < \infty .$$
\end{lemma}

\begin{proof}
By Proposition \ref{driftbound_choice}, $\sqrt{N}|\beta(Y_{s}^{N})-b(Y_{s}^{N})|=O(1)$, then
\begin{equation}
\begin{split}
|D_{t}^{N}|&\leq |D_{0}^{N}|+\sqrt{N}|M_{t}^{N}|+O(1)t+\int_{0}^{t}\sqrt{N} |b(Y_{s}^{N})-b(y_s)|ds\\
&\leq|D_{0}^{N}|+\sqrt{N}|M_{t}^{N}|+O(1)t+\int_{0}^{t}\sqrt{N}L|Y_{s}^{N}-y_s|ds\\
&=|D_{0}^{N}|+\sqrt{N}|M_{t}^{N}|+O(1)t+\int_{0}^{t}L|D_{s}^{N}|ds.
\end{split}
\end{equation}
By Gronwall's Lemma,
$$\sup_{0\leq t\leq T}|D_{t}^{N}|\leq e^{LT}\left(|D_{0}^{N}|+O(1)T+\sup_{0\leq t\leq T}|\sqrt{N}M_{t}^{N}|\right),$$
then
$$\limsup_{N\rightarrow \infty}\mathbb{E}\left(\sup_{0\leq t \leq T}|D_{t}^{N}|^2\right)\leq e^{2LT}\left[\limsup_{N\rightarrow \infty}\mathbb{E}(|D_{0}^{N}|)+O(1)T+\limsup_{N\rightarrow \infty}\mathbb{E}\left(\sup_{0\leq t\leq T}\sqrt{N}|M_{t}^{N}|\right)\right]^2.$$
We know that 
$$\left[\mathbb{E}\left(\sup_{0\leq t\leq T}\sqrt{N}|M_{t}^{N}|\right)\right]^2\leq N\mathbb{E}\left(\sup_{0\leq t\leq T}|M_{t}^{N}|^2\right)\leq 4NA(N)T,$$
and that $A(N)=O(\frac{1}{N})$. Therefore
$$\limsup_{N\rightarrow \infty}\mathbb{E}\left(\sup_{0\leq t\leq T}\sqrt{N}|M_{t}^{N}| \right)<\infty.$$
Together with our assumption $\limsup_{N\rightarrow \infty}\mathbb{E}(|D_{0}^{N}|^2)<\infty$, we have
$$\limsup_{N\rightarrow \infty}\mathbb{E}\left(\sup_{0\leq t \leq T}|D_{t}^{N}|^2 \right)<\infty.$$\\
\end{proof}

\begin{lemma}\label{tightness_choice}
If $(D_{0}^{N})_{N=1}^{\infty}$ is tight then $(D^{N})_{N=1}^{\infty}$ is tight and its limit points are continuous.
\end{lemma}
\begin{proof}
To prove the tightness of $(D^{N})_{N=1}^{\infty}$ and the continuity of the limit points, we only need to show the following two conditions holds for each $T>0$ and $\epsilon>0$,
\begin{itemize}
\item[(i)] 
\begin{equation}
\lim_{K\rightarrow \infty}\limsup_{N\rightarrow \infty}\mathbb{P}\left(\sup_{0\leq t\leq T}|D_{t}^{N}|>K \right)=0,
\end{equation}
\item[(ii)] 
\begin{equation}
\lim_{\delta\rightarrow 0}\limsup_{N\rightarrow \infty}\mathbb{P}\left(w(D^{N},\delta,T)\geq \epsilon \right)=0
\end{equation}
\end{itemize}
where for $x\in \mathbb{D}^{d}$,
\begin{equation}
w(x,\delta,T)=\sup\left\{\sup_{u,v\in[t,t+\delta]}|x(u)-x(v)|:0\leq t\leq t+\delta\leq T\right\}.
\end{equation}
By Lemma \ref{L2bound_choice}, there exists $C_{0}>0$ such that
\begin{eqnarray}
\lim_{K\rightarrow \infty}\limsup_{N\rightarrow \infty}\mathbb{P} \left(\sup_{0\leq t\leq T}|D_{t}^{N}|>K \right) &\leq& \lim_{K\rightarrow \infty}\limsup_{N\rightarrow \infty}\frac{\mathbb{E}\left(\sup_{0\leq t\leq T}|D_{t}^{N}|^2 \right)}{K^2} \nonumber\\
&\leq&  \lim_{K\rightarrow \infty}\frac{C_{0}}{K^2} \nonumber\\
&=&0,
\end{eqnarray}
which proves condition (i).

For condition (ii), we have that

\begin{eqnarray}
D^N_{u} - D^N_{v} &=& \underbrace{\sqrt{N} \cdot ( M^N_{u} - M^N_{v})}_{\text{first term}} + \underbrace{\int^{u}_{v} \sqrt{N} \left( \beta(Y^N_{z}) - b(Y^N_{z})  \right) dz}_{\text{second term}} \nonumber \\&+& \underbrace{\int^{u}_{v} \sqrt{N} \left( b(Y^N_{z}) - b(y_{z})  \right) dz }_{\text{third term}}
\end{eqnarray}
for any $0<t\leq u<v\leq t+\delta\leq T$.  Now it suffices to show that each of the three terms of $D^N_{u} - D^N_{v}$ satisfies condition (ii).  In what follows, we will show that each of the three terms satisfies condition (ii) to complete the proof of tightness. 

 For the first term, 
similar to the proof of Proposition ~\ref{drift}, we can show that 
$$\sup_{t\leq T}\left|\beta_{+}(Y_{t}^{N})-b_{+}(Y_{t}^{N})\right|\xrightarrow{p}0, \quad \sup_{t\leq T}\left|\beta_{-}(Y_{t}^{N})-b_{-}(Y_{t}^{N})\right|\xrightarrow{p}0.$$
And by the proof of Proposition \ref{Lipschitz}, $b_{+}(y), b_{-}(y)$ are also Lipschitz with constant $L$, then by the fact that the composition of Lipschitz functions are also Lipschitz,
\begin{eqnarray}
\max\left\{\sup_{t\leq T}|b_{+}(Y_{t}^{N})-b_{+}(y_{t})|,\sup_{t\leq T}|b_{-}(Y_{t}^{N})-b_{-}(y_{t})|\right\}\leq L\sup_{t\leq T}|Y_{t}^{N}-y_{t}|.
\end{eqnarray}
By Theorem ~\ref{fluid_limit}, 
\begin{equation}
\sup_{t\leq T}|Y_{t}^{N}-y_{t}|\xrightarrow{p} 0.
\end{equation}
Thus, for any $\epsilon>0$,
\begin{eqnarray}
& &\lim_{N\rightarrow \infty}\mathbb{P}\left(\sup_{t\leq T}\left|\boldlangle \sqrt{N}M^N(k)\boldrangle_{t}- \boldlangle M(k)\boldrangle_{t}\right|>\epsilon\right)\nonumber \\
&=&\lim_{N\rightarrow \infty}\mathbb{P}\left(\sup_{t\leq T}\left|\int_{0}^{t}\left(\beta_{+}(Y_{s}^{N})+\beta_{-}(Y_{s}^{N})- b_{+}(y_{s})-b_{-}(y_{s})\right)ds\right|>\epsilon\right)\nonumber\\
&\leq & \lim_{N\rightarrow \infty}\mathbb{P}\left(\sup_{t\leq T}T\left|\beta_{+}(Y_{t}^{N})- b_{+}(Y^{N}_{t})\right|>\epsilon/3\right)+\lim_{N\rightarrow \infty}\mathbb{P}\left(\sup_{t\leq T}T\left|\beta_{-}(Y_{t}^{N})- b_{-}(Y^{N}_{t})\right|>\epsilon/3\right)\nonumber\\
& &+\lim_{N\rightarrow \infty}\mathbb{P}\left(\sup_{t\leq T}2LT\left|Y_{t}^{N}- y_{t}\right|>\epsilon/3 \right)\nonumber \\
&=& 0,
\end{eqnarray}
which implies
\begin{equation}
\sup_{t\leq T}\left|\boldlangle \sqrt{N}M^N(k)\boldrangle_{t}- \boldlangle M(k)\boldrangle_{t}\right|\xrightarrow{p} 0.
\end{equation}

We also know that the jump size of $D^{N}_{t}$ is $1/\sqrt{N}$, therefore
\begin{equation}
\lim_{N\rightarrow \infty}\mathbb{E}\left[\sup_{0<t\leq T}\left|M^{N}_{t}-M^{N}_{t-}\right| \right]=0.
\end{equation}
By Theorem 1.4 in Chapter 7 of \citet{Ethier2009},  $\sqrt{N}M^{N}_{t}$ converges to the Brownian motion $M_{t}$ in distribution in $\mathbb{D}(\mathbb{R}_{+},\mathbb{R}^{K+1})$. By Prohorov's theorem, $(\sqrt{N}M^{N})_{N=1}^{\infty}$ is tight. And since $M_{t}$ is a Brownian motion, its sample path is almost surely continuous.

For the second term, we have by Proposition \ref{driftbound_choice} that the quantity $ \sqrt{N} \left( \beta(Y^N_{z}) - b(Y^N_{z})  \right) $ is bounded for any value of $z\in [0,T]$.  Therefore, there exists some constant $C_{1}$ that does not depend on $N$ such that
\begin{equation}
\sup_{z\in [0,T]}\sqrt{N} \left| \beta(Y^N_{z}) - b(Y^N_{z})  \right|\leq C_{1}.
\end{equation}
Then
\begin{eqnarray}
& &\lim_{\delta\rightarrow 0}\lim_{N\rightarrow \infty}\mathbb{P}\left(\sup_{u,v\in [0,T],|u-v|\leq \delta}\int^{u}_{v} \sqrt{N} \left| \beta(Y^N_{z}) - b(Y^N_{z})  \right| dz > \epsilon \right)\nonumber \\ 
&\leq &  \lim_{\delta\rightarrow 0}\lim_{N\rightarrow \infty}\mathbb{P}\left(\delta \sup_{z\in [0,T]}\sqrt{N} \left| \beta(Y^N_{z}) - b(Y^N_{z})  \right|  > \epsilon \right)\nonumber \\
&\leq & \lim_{\delta\rightarrow 0}\mathbb{P}\left(\delta C_{1}  > \epsilon \right)\nonumber \\
&=& 0.
\end{eqnarray}
 Thus, we have proved the oscillation bound for the second term.  Finally for the third term we have that 
\begin{eqnarray}
\int^{u}_{v} \sqrt{N} \left| b(Y^N_{z}) - b(y_{z})  \right| dz & \leq& \int^{u}_{v} \sqrt{N} L\left| Y^N_{z} - y_{z} \right| dz\nonumber \\
&=& \int^{u}_{v} L \cdot \left|D^N_{z}\right| dz \nonumber\\\
&\leq & L\delta \sup_{t\in [0,T]}|D^{N}_{t}|.
\end{eqnarray}
By Lemma \ref{L2bound_choice},  
\begin{eqnarray}
& &\lim_{\delta\rightarrow 0}\lim_{N\rightarrow \infty}\mathbb{P}\left(\sup_{u,v\in [0,T],|u-v|\leq \delta}\int^{u}_{v} \sqrt{N} \left| b(Y^N_{z}) - b(y_{z})  \right| dz>\epsilon\right)\nonumber\\
&\leq &\lim_{\delta\rightarrow 0}\lim_{N\rightarrow \infty}\mathbb{P}\left(L\delta\sup_{t\in [0,T]}|D^{N}_{t}|>\epsilon\right)\nonumber\\
&\leq & \lim_{\delta\rightarrow 0}\lim_{N\rightarrow \infty}\frac{\mathbb{E}\left(\sup_{t\in [0,T]}|D^{N}_{t}|^2\right)}{(\epsilon/L \delta)^2}\nonumber \\
&\leq &\lim_{\delta\rightarrow 0}\frac{C_{0}(L\delta)^2}{\epsilon^2}\nonumber \\
&=& 0,
\end{eqnarray}
 which implies that the oscillation bound holds for the third term.  
\end{proof}

\begin{proof}[Proof of Theorem \ref{Lyapunov}]
To show that $g$ is a Lyapunov function for the dynamical system $y(t)$, we only need to show that $\frac{d}{dt}g(y(t))=b(y)\nabla g(y)\leq 0$ and $\frac{d}{dt}g(y(t))=0$ if and only if $y=\bar{y}$.

Since
\begin{eqnarray}
g(y)&=&\sum_{n=0}^{K}y_n\log \left(\frac{y_n}{\nu_{\rho(y)}(n)}\right)-\log(Z(\rho(y))+\sum_{n=0}^{K}\sum_{k=0}^{n-1}S_{k,n}(y_n)\nonumber\\
&=&\sum_{n=0}^{K}y_n\log \left(\frac{y_n Z(\rho(y))}{\prod_{i=0}^{n-1}\rho_i(y)}\right)-\log(Z(\rho(y))+\sum_{n=0}^{K}\sum_{k=0}^{n-1}S_{k,n}(y_n)\nonumber\\
&=&\sum_{n=0}^{K}y_n\log \left(\frac{y_n }{\prod_{i=0}^{n-1}\rho_i(y)}\right)+\sum_{n=0}^{K}y_n\log(Z(\rho(y))-\log(Z(\rho(y))+\sum_{n=0}^{K}\sum_{k=0}^{n-1}S_{k,n}(y_n)\nonumber\\
&=&\sum_{n=0}^{K}y_n\left(\log (y_n)-\log \left(\prod_{i=0}^{n-1}\rho_i(y)\right)\right)+\sum_{n=0}^{K}\sum_{k=0}^{n-1}S_{k,n}(y_n).
\end{eqnarray}
Then we have the gradient of $g$ is
\begin{eqnarray}\label{gradient_g}
\frac{\partial g(y)}{\partial y_n}& =& \log(y_n)+1-\sum_{n=0}^{K}y_n\left(\sum_{i=0}^{n-1}\frac{\partial \rho_i(y)}{\partial y_n}/\rho_i(y)\right)-\log \left(\prod_{i=0}^{n-1}\rho_i(y)\right)\nonumber\\
& &+\sum_{n=0}^{K}\sum_{k=0}^{n-1}y_n\phi_{k,n}'(y_n)/\phi_{k,n}(y_n)\nonumber\\
&=& \log\left(\frac{y_n}{\prod_{i=0}^{n-1}\rho_i(y)}\right)+1.
\end{eqnarray}

Now denote $b(y)=yB_y$ where $B_y$ is the $(K+1)\times (K+1)$ generator matrix that depends on $y$. Due to the reversibility of generator $B_y$ with respect to distribution $\nu_{\rho(y)}$, we have
 $$\nu_{\rho(y)}(m)B_y(m,n)=\nu_{\rho(y)}(n)B_y(n,m).$$ 
 
 Denote $q_y(m,n)=\nu_{\rho(y)}(m)B_y(m,n)$ , then $q_y(m,n)$ is non-negative and symmetric in $(m,n)$. Then we can show that the Dirichlet form of associated to $B_y$ evaluated at vectors $\left(\frac{y_n}{\nu_{\rho(y)}(n)}\right)$ and any $u\in \mathbb{R}^{K+1}$ is equal to $-yB_y u$,
\begin{eqnarray}
& &\mathcal{E}\left(\left(\frac{y_n}{\nu_{\rho(y)}(n)}\right)_n, u\right)\nonumber\\
 &=& \frac{1}{2}\sum_{m,n=0}^{K}q_y(m,n)\left(\frac{y_m}{\nu_{\rho(y)}(m)}-\frac{y_n}{\nu_{\rho(y)}(n)}\right)(u_m-u_n)\nonumber\\
&=&\frac{1}{2}\sum_{m,n=0}^{K}\nu_{\rho(y)}(m)B_y(m,n)\left(\frac{y_m u_m}{\nu_{\rho(y)}(m)}+\frac{y_n u_n}{\nu_{\rho(y)}(n)}-\frac{y_mu_n}{\nu_{\rho(y)}(m)}-\frac{y_nu_m}{\nu_{\rho(y)}(n)}\right)\nonumber\\
&=&\frac{1}{2}\sum_{m,n=0}^{K}\left(B(m,n)y_mu_m+B(m,n)y_mu_m-B(m,n)y_mu_n-B(n,m)y_nu_m\right)\nonumber\\
&=&\sum_{m=0}^{K} \left(\sum_{n=0}^{K}B(m,n)\right)y_mu_m -\sum_{m,n=0}^{K}B(m,n)y_mu_n\nonumber\\
&=&-\sum_{m,n=0}^{K}B(m,n)y_mu_n\nonumber\\
&=&-yB_yu.
\end{eqnarray}
Here we use the fact that row sum of $B_y$ is zero. Then using Equation (\ref{gradient_g}), we have
\begin{equation}
yB_y \nabla g(y)=-\frac{1}{2}\sum_{m,n=0}^{K}q_y(m,n)\left(\frac{y_m}{\nu_{\rho(y)}(m)}-\frac{y_n}{\nu_{\rho(y)}(n)}\right)\left(\log\left(\frac{y_m}{\nu_{\rho(y)}(m)}\right)-\log\left(\frac{y_n}{\nu_{\rho(y)}(n)}\right)\right).
\end{equation}
This shows that $\frac{d }{dt}g(y(t))= yB_y \nabla g(y)\leq0$ for all $y$. It is zero only if $ \frac{y_m}{\nu_{\rho(y)}(m)}=\frac{y_n}{\nu_{\rho(y)}(n)} $ for all $m,n$ pairs that $q_y(m,n)>0$. By the irreducibility of $B_y$, this is only possible if $ \frac{y_n}{\nu_{\rho(y)}(n)}$ doesn't depend on $n$. This indicates $y=\nu_{\rho(y)}$, which is the equilibrium point of $y(t)$.
\end{proof}

\bibliographystyle{plainnat}
\bibliography{Bike_Choice}
\end{document}